\documentclass[12pt, reqno]{amsart}
\usepackage{amsmath,amsfonts, amsthm, amssymb, color, cite}
\usepackage{color}
\usepackage{fancyhdr}
\usepackage{graphicx}
\usepackage{epstopdf}
\usepackage{lastpage}
\newtheorem{thm}{Theorem}[section]
\newtheorem{lma}{Lemma}[section]
\newtheorem{pro}{Proposition}[section]

\theoremstyle{definition}

\theoremstyle{remark}
\newtheorem{Remark}{Remark}[section]

\numberwithin{equation}{section}
\allowdisplaybreaks

\def\di{\displaystyle}

\def\f{\frac}

\def\hf1{^\f{1}{1-\xi^2}}

\def\be{\begin{equation}}
\def\en{\end{equation}}
\def\bs{\begin{split}}
\def\es{\end{split}}
\def\ba{\begin{align}}
\def\ea{\end{align}}

\author[Xushan Huang ]
{Xushan Huang}
\address{\newline Institute of Applied Mathematics, Academy of Mathematics and Systems Science, Chinese Academy of Sciences, Beijing 100190, P. R. China
\newline and School of Mathematical Sciences, University of Chinese Academy of Sciences, Beijing 100049, P. R. China}
\email{huangxushan@amss.ac.cn}

\author[Yi Wang]
{Yi Wang}
\address{\newline Institute of Applied Mathematics, Academy of Mathematics and Systems Science, Chinese Academy of Sciences, Beijing 100190, P. R. China
\newline and School of Mathematical Sciences, University of Chinese Academy of Sciences, Beijing 100049, P. R. China}

\email{wangyi@amss.ac.cn}

\renewcommand{\fancyhead}{}

\title{Global existence and optimal time decay rate to one-dimensional two-phase flow model}
\keywords{Global existence, spectral analysis, weighted energy estimate, optimal time decay rate}
\date{\today}
\usepackage{subfig}
\usepackage{graphicx}

\thanks{\textbf{Acknowledgment.} The work of Yi Wang was partially supported by the NSFC (Grant No. 12171459, 12288201, 12090014) and CAS Project for Young Scientists in Basic Research, Grant No. YSBR-031.}
	
\begin{document}
	
\begin{abstract}
We investigate the global existence and optimal time decay rate of solution to the one-dimensional (1D) two-phase flow described by compressible Euler equations coupled with compressible Navier-Stokes equations through the relaxation drag force on the momentum equations (Euler-Navier-Stokes system). First, we prove the global existence of a strong solution and the stability of the constant equilibrium state to 1D Cauchy problem of compressible Euler-Navier-Stokes system by using the standard continuity argument for small $H^{1}$ data while its second order derivative can be large. Then we derive the optimal time decay rate to the constant equilibrium state. Compared with the multi-dimensional case, it is much harder to get the optimal time decay rate by the direct spectrum method due to a slower convergence rate of the fundamental solution in the 1D case. To overcome this main difficulty, we need to first carry out time-weighted energy estimates (not optimal) for higher order derivatives, and based on these time-weighted estimates, we can close a priori assumptions and get the optimal time decay rate by spectral analysis method. Moreover, due to the non-conserved form and insufficient decay rate of the coupled drag force terms between the two-phase flows, we essentially need to use momentum variables $(m= \rho u, M=n\omega)$, rather than velocity variables $(u, \omega)$ in the spectrum analysis, to fully cancel out those non-conserved and insufficient time decay drag force terms.

\end{abstract}

\maketitle
\tableofcontents
\section{Introduction and main results}\label{section 1}
\subsection{Introduction}
Two-phase flow models appear in a large number of important applications in nature and engineering \cite{be2003}. In this paper, we are concerned with the two-phase flow model 
described by the following compressible Euler equations coupled with compressible Navier-Stokes equations with density-dependent viscosities
\begin{equation}\label{tf}
\left\{
\begin{array}{l}
\di \rho_{t}+ \rm div_x (\rho u)=0,\\[2mm]
\di (\rho u)_t+ \rm div_x  (\rho u\otimes u) +\nabla_x p(\rho)=%div_x\mathbb{S} 
\rho n(\omega-u),\\[2mm]
\di n_{t}+ \rm div_x (n \omega)=0,\\[2mm]
\di (n \omega)_t+ \rm div_x  (n \omega \otimes \omega) +\nabla_x n=\rm div_x(n\,\mathbb{D}(\omega))-\rho n(\omega-u),
\end{array}
\right.  
\end{equation} 
through the relaxation drag force on the momentum equations, where the spatial variable $x\in \mathbb{R}^3$ and the time variable $t>0$ and  $\big(\rho(t,x) $, $u(t,x)\big)$ and $\big(n(t,x)$, $\omega(t,x)\big)$ are respectively the density and velocity of the two fluids, the pressure $p$ of one fluid is given by the standard $\gamma$-law: 
\begin{align}\label{eq:p}
	p(\rho)= a\rho^{\gamma},
\end{align}
 with $a>0$ being the fluid constant and $\gamma> 1$  being the adiabatic exponent. Note that one of the two fluids is isothermal  in \eqref{tf} and the pressure is linear to the density $n$ and its coefficient is normalized to be 1 without loss of generality. Moreover, $\mathbb{D}(\omega):=\frac{\nabla \omega+(\nabla \omega)^t}{2}$ is the deformation tensor with $(\nabla \omega)^t$ being the transpose of the matrix $\nabla \omega$, and its density-dependent viscosity coefficient is $n$.

The two-fluid system \eqref{tf} can be formally derived from Chapman-Enskog expansion of the fluid-particle model consisting of the compressible Euler equations for fluids coupled with the Vlasov-Fokker-Planck equation for particles, through the relaxation drag force on the momentum equation and the Vlasov force on the Fokker-Planck equation (cf. \cite{W2022}):
\begin{equation}\label{1.02}
\left\{\begin{array}{ll}
\rho_t + \rm div_x(\rho u) = 0,&\\
(\rho u)_t + \rm div_x(\rho u\otimes u) + \nabla_xp(\rho) = \di\int_{\mathbb{R}^3}\rho(v - u)f dv,&\\[3mm]
f_t + v\cdot \nabla_xf =\rm div_v\big(\rho(v-u)f + \nabla_vf\big).
\end{array} \right.
\end{equation}

Another common-used compressible two-phase flow model, coupled by isothermal Euler equations and classical isentropic Navier-Stokes equations (cf. \cite{choi2016}), can be written as

\begin{equation}\label{tf-1}
\left\{
\begin{array}{l}
\di \rho_{t}+ \rm div_x (\rho u)=0,\\[2mm]
\di (\rho u)_t+ \rm div_x  (\rho u\otimes u) +\nabla_x p(\rho)=\mu\Delta u+(\mu+\lambda)\nabla \rm div_x u+
\kappa n(\omega-u),\\[2mm]
\di n_{t}+ \rm div_x (n \omega)=0,\\[2mm]
\di (n \omega)_t+ \rm div_x  (n \omega \otimes \omega) +\nabla_x n=-\kappa n(\omega-u).
\end{array}
\right.  
\end{equation} 
%which can be formally derived from Hilbert expansion of the fluid-particle model consisting of the compressible Navier-Stokes equations for fluids coupled with the Vlasov-Fokker-Planck equation for particles:
%\begin{equation}
%	\left\{\begin{array}{ll}
	%	\di \rho_{t}+ div_x (\rho u)=0,\\[2mm]
	%	\di (\rho u)_t+ div_x  (\rho u\otimes u) +\nabla_x p(\rho)=\mu\Delta u+(\mu+\lambda)\nabla div_x u+
	%	\kappa n(\omega-u),\\[2mm]
	%	f_t + v\cdot \nabla_xf =div_{\omega}\big(\kappa(\omega-u)f + \nabla_{\omega} f\big).
%	\end{array} \right.
%	\end{equation} 

Note that both two-phase flow models \eqref{tf} and \eqref{tf-1} can be derived from the kinetic fluid-particle system and have similar couplings. However, the dissipation structures of two systems \eqref{tf} and \eqref{tf-1} are quite different, in particular, the viscosity of the momentum equation $\eqref{tf}_4$ depends on the density function linearly while the corresponding one of $ \eqref{tf-1}_4$ is ideal without viscosity.

For one fluid system described by compressible Navier-Stokes equations or Euler equations (possibly with the damping effect), there is a large amount of literature investigating the local and global well-posedness and large-time asymptotic behaviors of the solution, for instance, \cite{na1962, ser1959,  ka1977, cho2006, ma1980, hs1992, li1998, fe2001, jz-2001, hu2005, hlx-2012, yang2001} and the references therein.

For the compressible two-phase flow models \eqref{tf} or \eqref{tf-1}, there are relatively few results on the global existence and large-time behavior of the solution. In 2016, Choi \cite{choi2016} first established the global existence of a unique strong solution \eqref{tf-1} in three-dimensional domain $\mathbb{T}^{3}$ or $\mathbb{R}^{3}$, and further proved the exponential time-decay rate of the solution towards a constant equilibrium state in $\mathbb{T}^{3}$. Then Wu-Zhang-Zou \cite{wu2020} proved the optimal time decay rate of the solution to \eqref{tf-1} in $\mathbb{R}^3$ by direct spectral analysis method. Very recently, Li-Shou \cite{li2023-sima} proved the global existence of the solution to \eqref{tf-1} near the equilibrium state in a critical homogeneous Besov space, and obtained the optimal time-decay rates of the global solution toward the equilibrium state in $\mathbb{R}^d$ with $d\geq 2$. For the two-phase flow model \eqref{tf}, Wu-Zhang-Tang \cite{wu-zhang2023} proved the global existence of strong solution around the equilibrium state in $\mathbb{R}^{3}$ and derived the optimal algebraic time-decay rate in $L^2$-norm provided the classical viscosity is added to the compressible Euler part (i.e. classical compressible Navier-Stokes equations coupled with compressible Navier-Stokes equations with density-dependent viscosities).

However, there are no results on the optimal time decay rate of the solution to two-phase flow model \eqref{tf} or \eqref{tf-1} in one-dimensional case as far as we know. Compared with the multi-dimensional case, 1D fundamental solution has a slower convergence rate, which makes it hard to get 1D optimal time decay rate by direct spectrum method. 
In 1D setting, the compressible Euler-Navier-Stokes system \eqref{tf} takes the form
\begin{equation}\label{1.1}
\left\{\begin{array}{ll}
\rho_t + (\rho u)_x = 0,&\\[2mm]
(\rho u)_t + \left(\rho u^{2}+p(\rho)\right)_x = \rho n(\omega-u),&\\[2mm]
n_t + (n\omega)_x=0,&\\[2mm]
(n\omega)_t + (n\omega^{2}+n)_x =(n\omega_x)_x + \rho n(u-\omega),
\end{array} \right.
\end{equation}
where the pressure $p(\rho)$ is given by \eqref{eq:p}. 
Without loss of generality, we can normalize the fluid constant $a=1$ in the sequel.
The following initial data is imposed to the 1D system (\ref{1.1}):
\begin{equation}\label{1.2}
(\rho,u,n,\omega)(0,x)=(\rho_0,u_0,n_0,\omega_0)(x) \longrightarrow (\rho_{*}, 0, n_{*}, 0), ~x\rightarrow \pm\infty,
\end{equation}
for the constant equilibrium state $ (\rho_{*}, 0, n_{*}, 0)$ satisfying $\rho_{*}>0, n_{*}>0 .$
In the present paper, we are concerned with the global existence and optimal time decay rate to 1D two-phase flow model \eqref{1.1}-\eqref{1.2}.

First, we prove the global existence and the time-asymptotic stability of the constant equilibrium state $ (\rho_{*}, 0, n_{*}, 0)$ to 1D two-phase flow model \eqref{1.1}-\eqref{1.2} by the elementary entropy and energy methods under the smallness condition on the $H^{1}$ norm of the perturbation around a constant equilibrium state, while their second order derivatives can be large. Then we obtain the optimal time decay rate towards the constant equilibrium state $ (\rho_{*}, 0, n_{*}, 0)$ for the global solution to the 1D two-phase flow model \eqref{1.1}-\eqref{1.2}. Compared with the multi-dimensional case, the main difficulty for the optimal time decay rate in 1D case is due to the slower time decay rate of 1D fundamental solution, which makes it hard to get the time-decay rate by direct spectrum method. Motivated by the techniques in Zeng \cite{Z2019}, we first carry out the time-weighted energy estimates for higher-order derivatives of the solution, then based on these time decay estimates (not optimal) on the higher-order derivatives, we are able to deduce the optimal time decay rate through a spectral analysis approach. However, we need to emphasize that the two-phase flow model \eqref{1.1} does not satisfy the structure condition proposed in \cite{Z2019}. Moreover, due to the non-conserved form and insufficient decay rate of the coupled drag force terms between the two-phase flows, we essentially need to use momentum variables $(m= \rho u, M=n\omega)$, rather than velocity variables $(u, \omega)$ to the system \eqref{1.1} in the spectral analysis, to fully cancel out those non-conserved and insufficiently time-decay drag force terms.

%Additionally, we need to use the non-conserved nature and insufficient decay rate of the coupled drag force terms between the two-phase flow present additional complexities. To effectively address this, it becomes essential to employ momentum variables $(m= \rho u, M=n\omega)$ rather than velocity variables $(u, \omega)$ in order to fully counterbalance these non-conserved drag force terms.
%Moreover, due to non-conserved form and insufficient decay rate of the coupled drag force terms between the two-phase flow in the spectrum analysis, we essentially need to use momentum variables $(m= \rho u, M=n\omega)$, not velocity variables $(u, \omega)$, to fully cancel out those non-conserved drag force terms.

The rest of the paper is organized as follows. We first list some notations and an auxiliary lemma, then reformulate the problem and state our main results in Section \ref{section 1}. In Section \ref{section 2}, global existence and time-asymptotic stability of the equilibrium state $ (\rho_{*}, 0, n_{*}, 0)$ to \eqref{1.1}-\eqref{1.2} in Theorem \ref{thm 1.2} are proved by the continuity arguments, based on the local existence of the solution and the uniform-in-time a priori estimates. In Section \ref{section 3}, we first carry out the detailed spectral analysis of the linearized system for the variables $(\rho, m= \rho u, n, M=n\omega)$ around the equilibrium state $ (\rho_{*}, 0, n_{*}, 0)$ and then establish the time-weighted energy estimates (not optimal) to the derivatives of the perturbations for the nonlinear system. Finally, in Section \ref{section 4} we prove Theorem \ref{thm 1.3} for the optimal time-decay rate by combining the spectrum analysis and the time-weighted energy estimates obtained in Section \ref{section 3}.
%The linear system's  and the establishment of  can be found in , while  will encompass the derivation of the proof for .
%We organize the article as follows. In section \ref{section 1}, some notations, auxiliary lemmas and the main results would be provided. In section \ref{section 2}, we present the entropy estimate, a priori estimates and the proof of Theorem \ref{thm 1.2}. Spectrum analysis of the linear system, time weighted energy estimates are established in section \ref{section 3} and the proof of Theorem \ref{thm 1.3} would be obtained in section \ref{section 4}.

Now we list some notations and an auxiliary lemma.
If $f \in L^{p}(\mathbb{R})$, then the $L^{p}$ norm of $f$ is defined by 
$$ \|f\|_{L^{p}(\mathbb{R})} := \left(\int_\mathbb{R} |f(x)|^{p} dx\right)^{\frac{1}{p}}.$$
For $p=2$, we simply write $\|\cdot\|:=\|\cdot\|_{L^2(\mathbb{R})}$ and for $p=\infty$, $\|\cdot\|_{\infty}:=\|\cdot\|_{L^{\infty}(\mathbb{R})}$.

%The following 1D Gagliardo-Nirenberg inequality is often used, while its proofs will be omitted.
%Next we give two lemmas which play an important role in subsequent parts of this paper. The proof can be found in \cite{Z2015} and \cite{N1959} respectively.
\begin{lma}{\rm{(}}1D Gagliardo-Nirenberg inequality{\rm{)}}\label{lma 1.2}
Let $j$ and $m$ be non-negative integers such that $j<m$ and $1\leq p, q, r \leq +\infty$ be positive extended real numbers and $\theta \in [0,1]$ such that the relations 
$$\frac{1}{p} = \frac{j}{n} + \theta(\frac{1}{r}- m)+ \frac{1-\theta}{q},  \quad  \frac{j}{m} \leq \theta \leq 1 $$
hold. Then, 
\begin{equation}\label{GN}
\|\partial_x^{j}u\|_{L^{p}(\mathbb{R})}\leq C\|\partial_x^{m}u\|^{\theta}_{L^{r}(\mathbb{R})}\|u\|^{1-\theta}_{L^{q}(\mathbb{R})}
\end{equation}
for any $u\in L^{q}(\mathbb{R})$ such that $\partial_x^{m}u \in L^{r}(\mathbb{R})$, where the positive constant $C$ only depends on the parameters $j, m, q, r, \theta$. In particular, for $f\in H^{1}(\mathbb{R})$, 
\begin{equation}\label{Sob}
\|f\|_{\infty} \leq C\|f_x\|^{\frac{1}{2}}\|f\|^{\frac{1}{2}}.
\end{equation}
\end{lma}
\subsection{Reformulation of the problem and main results}
Denote the sound speed 
$ \sigma (\rho) := \sqrt{p'(\rho)} $  and set $ \sigma_{*} := \sigma (\rho
_{*}) $. Define (cf. \cite{W2003})
\begin{equation}
v := \frac{2}{\gamma -1}  (\sigma (\rho) - \sigma_{*}), 
\end{equation}
then the system \eqref{1.1} can be transformed into the following symmetric system:
\begin{equation}\label{2.2}
\left\{\begin{array}{l}
\di v_t + \sigma_{*}u_x = -uv_x - \frac{\gamma -1}{2}vu_x,\\[3mm]
\di u_t + \sigma_{*}v_x = -uu_x - \frac{\gamma -1}{2}vv_x + n(\omega-u),\\[3mm]
\di n_t+(n\omega)_x=0,\\[2mm]
\di (n\omega)_t+(n\omega^{2}+n)_x=(n\omega_x)_x+\rho n(u-\omega),
\end{array} \right.
\end{equation}
with the initial condition
\begin{equation}\label{2.3}
(v, u, n, \omega)|_{t=0} = (v_0(x), u_0(x), n_0(x), \omega_0(x)).
\end{equation}

Note that for the classical solution, the systems \eqref{1.1} and \eqref{2.2} are equivalent. To state our main results, we first define the solution space to the system \eqref{1.1} or \eqref{2.2} on the time interval $[0,T]$ by 
\begin{equation*}
\begin{aligned}
X(0,T) = & \bigg\{(\rho, u, n, \omega) \bigg|(\rho -\rho_*, u, n-n_*) \in \Big(C\big(0,T;H^{2}(\mathbb{R})\big)\cap L^{2}\big(0,T;H^{2}(\mathbb{R})\big)\Big)^{3}, \\
& \qquad\qquad\qquad\omega \in 
C\big(0,T;H^{2}(\mathbb{R})\big)\cap L^{2}\big(0,T;H^{3}(\mathbb{R})\big)\bigg\} .
\end{aligned}
\end{equation*}
Our first main result concern the global existence and time-asymptotic stability of the equilibrium state $ (\rho_{*}, 0, n_{*}, 0)$ to 1D compressible Euler-Navier-Stokes equations \eqref{1.1}-\eqref{1.2}. %under the smallness condition on the $H^{1}$ norm of the perturbation around a constant equilibrium state, while their second order derivatives can be large.

\begin{thm}{\rm{(}}Global existence{\rm{)}}\label{thm 1.2}
Assume $(\rho_0-\rho_*, u_0, n_0 - n_{*}, \omega_0)\in H^{2}(\mathbb{R})$ and satisfies that $ \mathcal{E}_0:=\|(\rho_{0xx}, u_{0xx}, n_{0xx}, \omega_{0xx})\|   < \min\{\frac{1}{2} \mathcal{E}, \frac{1}{4\sqrt{C_0}} \mathcal{E}\}$ 
and $\epsilon_0:= \|(\rho_0-\rho_*, u_0, n_0 - n_{*}, \omega_0)\|_{H^{1}}  <\min\{\frac{1}{2}\epsilon, \frac{1}{2\sqrt{C_0}}\epsilon, \mathcal{E}_0\}$ %is a positive constant. 
with  $\mathcal{E}, \epsilon$ and $C_0$ being defined in Proposition \ref{a-priori}.
Then the two-phase flow system (or compressible Euler-Navier-Stokes equations) \eqref{1.1}-\eqref{1.2} has a unique global-in-time solution $(\rho, u, n, \omega)\in X(0,\infty)$ satisfying
%There exist a suitably small positive constant $\epsilon_{0}$ and a positive constant $C_{0}$ such that if the initial value \eqref{1.2} satisfies $ \|(\rho_0-\rho_*, u_0, n_0 - n_{*}, \omega_0)\|_{H^{1}} \leq \epsilon_0 $ and $ \|(\rho_{0xx}, u_{0xx}, n_{0xx}, \omega_{0xx})\| \leq {\mathcal{E}_0} $,  then the two phase flow system (or compressible Euler-Navier-Stokes equations) \eqref{1.1}-\eqref{1.2} has a unique global-in-time solution $(\rho, u, n, \omega)\in X(0,\infty)$ satisfying
\begin{equation}\label{11111}
\begin{aligned}
& \sup\limits_{t\in [0,\infty)}\big[\|(\rho -\rho_*, u, n-n_*, \omega)(t,\cdot)\|_{H^{1}}^{2} +\|(\rho_t, u_t)(t,\cdot)\|^{2}\big]  \\
& + \int_{0}^{\infty}\big[\|(\rho_x, u_x, n_x, u_t)(t,\cdot)\|^{2} +\|\omega_{xx}(t,\cdot)\|^{2} + \|(u-\omega)(t,\cdot)\|^{2}\big] dt  \leq C_0\epsilon_0^{2},  \\
& \sup\limits_{t\in [0,\infty]}\|(\rho_{xx}, u_{xx},n_{xx}, \omega_{xx}, \rho_{xt}, u_{xt})(t,\cdot)\|^{2} \\
& + \int_{0}^{\infty} (\|(\rho_{xx}, u_{xx}, n_{xx}, u_{xt})(t,\cdot)\|^{2} + \|\omega_{xxx}(t,\cdot)\|^{2} )dt \leq C_0(\epsilon_{0}^{2}+{\mathcal{E}^2_0}).
\end{aligned}
\end{equation}
%where $C_0$ is a uniform-in-time positive constant defined in Proposition \ref{a-priori}. 
Consequently, the following time-asymptotic behavior of the solution holds true,
\begin{equation}
\|(\rho-\rho_*, u, n-n_*, \omega)(t,\cdot)\|_{\infty} + \|(\rho_x, u_x, n_x, \omega_x)(t,\cdot)\|_{H^{1}} \longrightarrow 0 \quad as  \quad t\rightarrow \infty.
\end{equation}
\end{thm}
%For the following Theorem, optimal time decay rate, we apply the same notations $\epsilon_0, \mathcal{E}_0$ as in Theorem \ref{thm 1.2}. In fact, they are same when their derivatives are in the same order.
 Furthermore, we have the following theorem for the optimal time decay.

\begin{thm}{\rm{(}}Optimal time decay rates{\rm{)}}\label{thm 1.3}
Assume $(\rho_0-\rho_*, u_0, n_0 - n_{*}, \omega_0)\in H^{4}(\mathbb{R})$ satisfying $ F_0:= \|(\rho_0-\rho_*, u_0, n_0 - n_{*}, \omega_0)\|_{H^{1}} +\|(\rho_0-\rho_*, u_0, n_0-n_*, \omega_0)\|_{L^{1}} = \epsilon_0 + \|(\rho_0-\rho_*, u_0, n_0-n_*, \omega_0)\|_{L^{1}} $ suitably small and $\mathcal{E}_0:=\|(\rho_0, u_0, n_0, \omega_0)_{xx}\|_{H^{2}} <\infty$, then the global solution obtained in Theorem \ref{thm 1.2} has the following optimal time decay rate:
\begin{equation}\label{1.6}
\left\{\begin{array}{ll}
\|\partial_x^{k}(\rho - \rho_*, u, n-n_*, \omega)(t,\cdot)\| \leq CF_0(1+t)^{-\frac{1}{4} - \frac{k}{2}}, \quad k=0, 1, &\\
\|\partial_x^{2}(\rho - \rho_*, u, n-n_*, \omega)(t,\cdot)\| \leq C{\mathcal{E}_0}(1+t)^{-\frac{5}{4}},
\end{array} \right. 
\end{equation}
and 
\begin{equation}\label{1.8}
\|\partial_x^{k}(u- \omega)(t,\cdot)\| \leq C(F_0 +{\mathcal{E}_0}) (1+t)^{-\frac{3}{4}- \frac{k}{2}}, \quad  k=0, 1.
\end{equation}
\end{thm}
\begin{Remark}
The time decay rate in \eqref{1.6} for 1D Euler-Navier-Stokes equations \eqref{1.1} is the same as the classical 1D heat equation, even though the Euler part has only weak relaxation drag force but without strong dissipation effect. Moreover, \eqref{1.8} shows that the velocity difference $u-\omega$ exhibits a faster time decay rate than $u$ or $\omega$ itself.
%and in this sense, we say it is optimal.
%From (\ref{1.6}), we can see that the time decay rate of the solution of one-dimensional Euler-Navier-Stokes equations is the same as that of one-dimensional heat conduction equation. In this sense, the time decay estimate obtained by this theorem is optimal.
\end{Remark}
\begin{Remark}
%In contrast to the cases in higher dimensions (two and three dimensions), the time decay analysis established by this theorem cannot be directly attained through spectral analysis. This disparity arises due to the fact that, in the context of the one-dimensional scenario, the fundamental solution of the corresponding system experiences a comparatively slower rate of time decay. As a result, our approach involves a twofold strategy:
%Initially, we employ the time-weighted energy method to derive the time decay rate for the higher-order derivatives, albeit non-optimally. Subsequently, this outcome is amalgamated with the spectral analysis technique. This combined methodology yields the optimal time decay estimation for the solutions of the one-dimensional Euler-Navier-Stokes equations. The nuanced intricacies of the one-dimensional case necessitate this hybrid approach to ascertain the desired time decay estimates.

Compared with the multi-dimensional case (2D/3D), it is much harder to get the optimal time decay rate by direct spectrum method due to a slower convergence rate of the fundamental solution in the 1D case. To overcome this main difficulty, it is essential to first carry out time-weighted energy estimates (not optimal) for higher order derivatives, and based on these time-weighted estimates, we can close a priori assumptions and get the optimal time decay rate by spectrum analysis method. Moreover, due to the non-conserved form and insufficient decay rate of the coupled drag force terms between the two-phase flows, we essentially need to use momentum variables $(m= \rho u, M=n\omega)$, rather than velocity variables $(u, \omega)$ in the spectrum analysis, to fully cancel out those non-conserved and insufficiently time-decay drag force terms.

%Compared with the higher dimensional (two and three dimensional) cases, the time decay estimation of this theorem cannot be obtained by direct spectral analysis, because in the one-dimensional case, the fundamental solution of the one-dimensional system has a lower time decay speed. Therefore, we first use the time-weighted energy method to obtain the non-optimal time decay velocity of the higher order derivative, and then combine it with the spectral analysis method to obtain the optimal time decay estimation of the solution of one-dimensional Euler-Navier-Stokes equations.
\end{Remark}

\section{The proof of Theorem \ref{thm 1.2}}\label{section 2}

In this section, we will prove Theorem \ref{thm 1.2}. First, we give the local existence of the solution to the system \eqref{1.1} or \eqref{2.2}.

%Proposition \ref{thm 1.1} establishes the local existence of solutions for the Cauchy problem \eqref{2.2} and \eqref{2.3}. In the context of the Euler equations, this can be derived by employing the methods presented in \cite{K1975}. Similarly, for the compressible Navier-Stokes equations, this result is evidenced in \cite{C2019}. In the case of our model, which comprises both Euler and Navier-Stokes equations, the attainment of this proposition relies on utilizing interaction-based arguments.
%Proposition \ref{thm 1.1} is the local existence to the Cauchy problem \eqref{2.2} and \eqref{2.3}. For Euler equations, this can be obtained by using the arguments in \cite{K1975}, for compressible Navier-Stokes equations, this can be seen in \cite{C2019}. For our model which consist of Euler equations and Navier-Stokes equations, this can be achieved by using the interaction arguments. 

\begin{pro} {\rm{(}}Local existence{\rm{)}}\label{thm 1.1}
Assume $(\rho_0-\rho_*, u_0, n_0 - n_{*}, \omega_0) \in H^{2}(\mathbb{R})$, and satisfying $\|(\rho_0-\rho_*, u_0, n_0 - n_{*}, \omega_0)\|_{H^{1}} \leq P_{1} $, $ \|(\rho_{0}, u_{0},  n_{0}, \omega_{0})_{xx}\| \leq P_{2} $ for $P_1, P_2$ are two positive constants, then there exists a positive constant $T_0=T_0(P_1, P_2)$ such that
the system \eqref{1.1}-\eqref{1.2} (or \eqref{2.2}-\eqref{2.3}) has a unique solution $(\rho, u, n, \omega)\in X(0,T_0)$ satisfying
$$\sup\limits_{t\in [0,T_0]}\|(\rho-\rho_*, u, n - n_{*}, \omega)(t,\cdot)\|_{H^{1}} \leq 2P_{1} $$  
and 
$$\sup\limits_{t\in [0,T_0]}\|(\rho, u, n, \omega)_{xx}(t,\cdot)\| \leq 2P_{2}. $$
\end{pro}

The proof of Proposition \ref{thm 1.1} is standard and will be omitted for brevity. To prove 
Theorem \ref{thm 1.2}, it is sufficient to deduce the following uniform-in-time a priori estimate.

%For the common constant $C$, define 
%$$ {\epsilon} \leq \frac{1}{32\sqrt{C}(1+{\mathcal{E}_0})}, \quad \text{and}\quad {\mathcal{E}} = \min\left\{2\sqrt{C}(1 + {\mathcal{E}_0}),\frac{1}{16\epsilon_0}\right\}.$$
\begin{pro}  {\rm{(}}A priori estimate{\rm{)}}\label{a-priori}
Suppose that $(\rho, u, n, \omega) \in X(0,T) $ is a solution of \eqref{1.1}-\eqref{1.2} (or \eqref{2.2}-\eqref{2.3})  for some given $T>0$. There exists a positive constant $\mathcal{E}$
such that
%If the solution $(\rho, u, n, \omega)$ satisfies
\begin{align*}
	\sup\limits_{t\in [0,T]} \|(\rho_{xx}, u_{xx},n_{xx},\omega_{xx})(t,\cdot)\| \leq \mathcal{E}
\end{align*}
% for some  and  %related to $\mathcal{E}$ 
and a suitably small positive constant $\epsilon$ such that
\begin{equation}\label{assumption}
\begin{aligned}
\sup\limits_{t\in [0,T]} \|(\rho -\rho_*,u,n-n_{*},\omega)(t,\cdot)\|_{H^{1}(\mathbb{R})} \leq \epsilon,
\end{aligned}
\end{equation}
and  satisfying $\epsilon \mathcal{E} \leq \frac{1}{16}$,
then there exists a generic uniform positive constant $C_0$ such that the following a priori estimates hold uniformly for $t\in [0,T]$: 
\begin{equation}
\begin{aligned}
& \|(\rho -\rho_*, u, n-n_*, \omega)(t,\cdot)\|_{H^{1}}^{2} +\|(\rho_t, u_t)(t,\cdot)\|^{2}  \\
& + \int_{0}^{t}\big[\|(\rho_x, u_x, n_x, u_{\tau})(\tau,\cdot)\|^{2} +\|\omega_{xx}(\tau,\cdot)\|^{2} + \|(u-\omega)(\tau,\cdot)\|^{2}\big] d\tau  \leq C_0\epsilon_0^{2},  \\
& \|(\rho_{xx}, u_{xx},n_{xx}, \omega_{xx}, \rho_{xt}, u_{xt})(t,\cdot)\|^{2} \\
& + \int_{0}^{t} (\|(\rho_{xx}, u_{xx}, n_{xx}, u_{x\tau})(\tau,\cdot)\|^{2} + \|\omega_{xxx}(\tau,\cdot)\|^{2} )d\tau \leq C_0(\epsilon_{0}^{2}+{\mathcal{E}^2_0}).
\end{aligned}
\end{equation}

\end{pro}

The proof of Proposition \ref{a-priori} occupies the main part of the remaining of the present section. First, we have

%\subsection{Entropy Estimate and A Priori Estimates}
\begin{lma}{\rm{(}}Entropy estimate{\rm{)}}\label{thm 2.1}
Under the hypothesis and a priori assumptions in Proposition \ref{a-priori}, there exists $C>0$ such that the following estimate holds:
\begin{equation}
\begin{aligned}
&\|(\rho-\rho_*,u,n-n_{*},\omega)(t, \cdot)\|^2+\int_0^t\|(\omega_x, u-\omega)(\tau, \cdot)\|^2d\tau  \leq C\epsilon_0^{2}.
\end{aligned}
\end{equation}
\end{lma}
\begin{proof}\color{black}
By virtue of $\eqref{1.1}_1$ and $\eqref{1.1}_2$,
we can get
\begin{equation}\label{2.04}
\rho u_t+\rho u u_x+p(\rho)_x=\rho n(\omega-u).
\end{equation}
Multiplying \eqref{2.04} by $u$ and using $\eqref{1.1}_1$ yield
\begin{equation}\label{2.09}
\left(\rho \frac{u^2}{2}\right)_t+\left(\rho u \frac{u^2}{2}\right)_x+\left(\rho^\gamma-\rho_{*}^\gamma\right)_x u=\rho n(w-u)u.
\end{equation}
Multiplying $ \eqref{1.1}_1 $ by $ \frac{\gamma}{\gamma - 1}(\rho^{\gamma
	 -1} - \rho_{*}^{\gamma - 1}) $, we have
\begin{equation}\label{2.014}
Q_1(\rho,\rho_*)_t+\left[Q_1(\rho,\rho_*) u\right]_x + \left( \rho^\gamma - \rho_*^\gamma \right) u_x=0,
\end{equation}
where
$$
Q_1(\rho,\rho_*)=\frac{1}{\gamma-1}\left[\rho^{\gamma}-\rho_{*}^{\gamma}-\gamma\rho_*^{\gamma-1}(\rho-\rho_*)\right].
$$
Combining \eqref{2.09} and \eqref{2.014} gives
\begin{equation}\label{2.1}
\begin{aligned}
&\left[\rho \frac{u^{2}}{2} + Q_1(\rho,\rho_*)\right]_t
 + \left[\rho u\frac{u^{2}}{2} + Q_1(\rho,\rho_*) u+(\rho^{\gamma}-\rho_{*}^{\gamma})u \right]_x = \rho n(\omega-u)u.
\end{aligned}
\end{equation}
Similarly, by $\eqref{1.1}_3$ and $\eqref{1.1}_4$, we have
\begin{equation}\label{2.012}
\begin{aligned}
n\omega_t + n\omega\omega_x + n_x = (n\omega_x)_x + \rho n(u-\omega).
\end{aligned}
\end{equation}
Multiplying \eqref{2.012} by $ \omega $ and using $\eqref{1.1}_3$ yield
\begin{equation}\label{2.7}
\begin{aligned}
\left(n \frac{\omega^{2}}{2}\right)_t + \left(n\omega \frac{\omega^{2}}{2}\right)_x + (n-n_{*})_x\omega - (n\omega \omega_x)_x + n\omega_x^{2} =  \rho n(u - \omega) \omega.
\end{aligned}
\end{equation}
Multiplying $\eqref{1.1}_3$ by $ [(1 + \ln n) - (1+ \ln n_{*})]$ gives
\begin{equation}\label{2.11}
Q_2(n,n_*)_t+\left[Q_2(n,n_*) \omega\right]_x + \left( n- n_* \right) \omega_x=0,
\end{equation}
where
$$
Q_2(n,n_*)=n \ln n - n_*\ln n_{*} - (1 + \ln n_{*})(n - n_{*}).
$$
By \eqref{2.7} and \eqref{2.11}, we have
\begin{equation}\label{2.13}
\big[n\frac{\omega^{2}}{2} + Q_2(n,n_*) \big]_t + \left[n\omega \frac{\omega^{2}}{2}+Q_2(n,n_*) \omega+(n - n_{*})\omega \right]_x +n \omega_x^{2} = \rho n(u-\omega)\omega.
\end{equation}
Combining \eqref{2.1} and \eqref{2.13} and then integrating over $\mathbb{R}\times[0,t]$, we obtain
\begin{equation}
\begin{aligned}
&\int_{\mathbb{R}}\big[\rho \frac{u^{2}}{2} + Q_1(\rho,\rho_*)
+ n\frac{\omega^{2}}{2}+ Q_2(n,n_*) \big](t,x) dx \\
& + \int_{0}^{t}\int_{\mathbb{R}} \big[n\omega_x^{2} +\rho n(u-\omega)^{2}\big] dxd\tau \\
& = \int_{\mathbb{R}} [\rho_{0} \frac{u_{0}^{2}}{2} + Q_1(\rho_0,\rho_*)
+ n_0\frac{\omega_0^{2}}{2}+ Q_2(n_0,n_*) ] dx.
 \end{aligned}
\end{equation}
Consequently, the proof of Proposition \ref{thm 2.1} is completed.
\end{proof}

\subsection{Estimate of $ n_x $}
\begin{lma}\label{thm 2.2}
Under the hypothesis and a priori assumptions in Proposition \ref{a-priori}, there exists $C>0$ such that the following estimate holds:
\begin{equation}
\|n_x(t,\cdot)\|^{2}+ \int_{0}^{t} \|n_x(\tau,\cdot)\|^{2}d\tau \leq C \epsilon_{0}^{2}.
\end{equation}
\end{lma}
\begin{proof}
By $ \eqref{1.1}_3 $, we have
\begin{equation}\label{2.21}
n(\ln n)_{xt} + n\omega (\ln n)_{xx} + (n\omega_x)_x = 0,
\end{equation}
which together with \eqref{2.012} yield that
\begin{equation}\label{2.23}
n\big(\omega + (\ln n)_x\big)_t + n\omega \big(\omega +(\ln n)_x\big)_x + n_x = \rho n(u-\omega).
\end{equation}
Multiplying \eqref{2.23} by $ \omega + (\ln n)_x $ and using $\eqref{1.1}_3$, we obtain
\begin{equation}\label{2.26}
\left[n \frac{(\omega + (\ln n)_x)^{2}}{2}\right]_t + \left[n\omega \frac{(\omega + (\ln n)_x)^{2}}{2}\right]_x + \omega n_x + \frac{n_x^{2}}{n}= \rho n(u-\omega)[\omega + (\ln n)_x],
\end{equation}
which together with \eqref{2.11} and \eqref{2.1} implies that
\begin{equation}\label{2.28}
\begin{array}{ll}
\di \left[n \frac{(\omega + (\ln n)_x)^{2}}{2}+Q_2(n,n_*)+\rho \frac{u^{2}}{2} + Q_1(\rho,\rho_*)\right]_t + \left[n\omega \frac{(\omega + (\ln n)_x)^{2}}{2}\right.
\\[3mm]
\di  \left.
+ Q_2(n,n_*) \omega+\rho u\frac{u^{2}}{2} + Q_1(\rho,\rho_*) u+(\rho^{\gamma}-\rho_{*}^{\gamma})u \right]_x + \left[( n- n_*) \omega\right]_x + \frac{n_x^{2}}{n}\\[3mm]
\di +\rho n(u-\omega)^2 =\rho n(u-\omega) n_x.
\end{array}
\end{equation}
Integrating \eqref{2.28} over $ \mathbb{R} \times [0,t]$ with respect to $x, t$ and using Cauchy inequality, we have 
\begin{equation}\label{2.31}
\begin{aligned}
& \int_{\mathbb{R}}\left[n \frac{(\omega + (\ln n)_x)^{2}}{2} + Q_2(n,n_*)+\rho \frac{u^{2}}{2} + Q_1(\rho,\rho_*)\right](t,x) dx   + \int_{0}^{t}\int_{\mathbb{R}}\frac{n_x^{2}}{n} dxd\tau \\
\leq&\int_{\mathbb{R}}\left[n_{0} \frac{(\omega_{0} + (\ln n_{0})_x)^{2}}{2} + Q_2(n_0,n_*)\right] dx + \int_{0}^{t}\int_{\mathbb{R}}\rho (u-\omega)n_x dxd\tau\\
&\leq \int_{\mathbb{R}}\left[n_{0} \frac{(\omega_{0} + (\ln n_{0})_x)^{2}}{2} + Q_2(n_0,n_*)\right] dx + C\int_{0}^{t}\|u-\omega\|\|n_x\|d\tau.
\end{aligned}
\end{equation}
By virtue of Young's inequality and Lemma \ref{thm 2.1}, we can prove Lemma \ref{thm 2.2}.
 \end{proof}

\subsection{Estimate of $ \omega_x $}
\begin{lma}\label{thm 2.3}
Under the hypothesis and a priori assumptions in Proposition \ref{a-priori}, there exists $C>0$ such that the following estimate holds:
\begin{equation}
\|\omega_x(t,\cdot)\|^{2} + \int_{0}^{t}\|\omega_{xx}(\tau,\cdot)\|^{2} d\tau \leq C\epsilon_0^{2}.
\end{equation}
\end{lma}
\begin{proof}
Dividing \eqref{2.012} by $n$ and multiplying the resulted equation by $ -\omega_{xx} $ yield
\begin{equation}
-(\omega_x\omega_t)_x + \omega_x \omega_{xt} - \omega\omega_x\omega_{xx} + (\ln n)_x\omega_{xx} = -\omega_{xx}^{2} +(\ln n)_x\omega_x\omega_{xx} +\rho (u-\omega)\omega_{xx}.
\end{equation}
Integrating over $\mathbb{R} \times (0,t)$ implies that 
\begin{equation}\label{2.40}
\begin{aligned}
& \int_{\mathbb{R}} \frac{\omega_x^{2}}{2}dx + \int_{0}^{t}\int_{\mathbb{R}} \omega_{xx}^{2}dxd\tau = \int_{\mathbb{R}} \frac{\omega_{0x}^{2}}{2}dx -\int_{0}^{t} \int_{\mathbb{R}} (\ln n)_x\omega_x\omega_{xx} dxd\tau + \\
& \int_{0}^{t}\int_{\mathbb{R}} \omega\omega_x\omega_{xx}dxd\tau 
+ \int_{0}^{t}\int_{\mathbb{R}}(\ln n)_x\omega_{xx}dxd\tau + \int_{0}^{t}\int_{\mathbb{R}}\rho (u-\omega)\omega_{xx} dxd\tau.
\end{aligned}
\end{equation}
By \eqref{Sob} and Young's inequality, we obtain 
\begin{equation*}
\begin{aligned}
& \int_{0}^{t} \int_{\mathbb{R}} (\ln n)_x\omega_x\omega_{xx} dxd\tau \leq \int_{0}^{t} \|\omega_x\|_{\infty}\|\omega_{xx}\|\|(\ln n)_x\| d\tau \\
\leq&C \int_{0}^{t} \|\omega_x\|^{\frac{1}{2}}\|\omega_{xx}\|^{\frac{3}{2}} d\tau \leq \delta \int_{0}^{t}\|\omega_{xx}\|^{2}d\tau + C_\delta\int_{0}^{t} \|\omega_x\|^{2} d\tau,
\end{aligned}
\end{equation*}
where $\delta$ is a small positive constant to be determined and $C_\delta$ is a positive constant depending on $\delta$.
Similarly,
\begin{equation*}
\begin{aligned}
& \int_{0}^{t}\int_{\mathbb{R}} \omega\omega_x\omega_{xx}dxd\tau  \leq \delta \int_{0}^{t} \|\omega_{xx}\|^{2} d\tau + C_\delta
\int_{0}^{t} \|\omega_x\|^{2} d\tau,   \\
& \int_{0}^{t}\int_{\mathbb{R}}(\ln n)_x\omega_{xx}dxd\tau \leq \delta
 \int_{0}^{t} \|\omega_{xx}\|^{2} d\tau + C_\delta\int_{0}^{t} \|n_x\|^{2} d\tau,  \\
& \int_{0}^{t}\int_{\mathbb{R}}\rho (u-\omega)\omega_{xx} dxd\tau  \leq \delta \int_{0}^{t}  \|\omega_{xx}\|^{2} d\tau + C_\delta\int_{0}^{t}  \|u - \omega\| ^{2} d\tau.
\end{aligned}
\end{equation*}
Plugging the above estimates into \eqref{2.40}, taking $\delta$  suitably small and using Proposition \ref{thm 2.1} and Lemma \ref{thm 2.2}, we can deduce Lemma \ref{thm 2.3}.
\end{proof}
\subsection{Estimate of $ v_x, u_x, v_t, u_t $}
\begin{lma}\label{thm 2.4}
Under the hypothesis and a priori assumptions in Proposition \ref{a-priori}, there exists $C>0$ such that the following estimate holds:
\begin{equation}
\|(v_x, u_x, v_t, u_t)(t,\cdot)\|^{2} + \int_{0}^{t} \|(u_x, u_{\tau})(\tau,\cdot)\|^{2} d\tau \leq C\epsilon_{0}^{2}.
\end{equation}
\end{lma}
\begin{proof}
Differentiating $\eqref{2.2}_1,\eqref{2.2}_2$ with $x$, multiplying them by $ v_x, u_x $ respectively, then adding them together and integrating over $\mathbb{R} \times (0,t)$, we can get
\begin{equation}\label{2.62}
\begin{aligned}
&\frac{d}{dt} \int_{\mathbb{R}}\frac{1}{2}(v_x^{2} + u_x^{2}) dx + \int_{\mathbb{R}} nu_x^{2} dx  \\
 = & -\frac{1}{2}\int_{\mathbb{R}} u_x^{3} dx  -\frac{\gamma}{2} \int_{\mathbb{R}} u_xv_x^{2} dx 
 + \int_{\mathbb{R}} n_x(\omega - u)u_x dx + \int_{\mathbb{R}} n\omega_xu_x dx.
\end{aligned}
\end{equation}
By using \eqref{Sob}, we obtain
\begin{equation*}
\left|\int_{\mathbb{R}} u_x^{3} dx\right| \leq C \|u_x\|_{\infty} \|u_x\|^{2} \leq C{\epsilon}^{\frac{1}{2}}{\mathcal{E}}^{\frac{1}{2}}\|u_x\|^{2}.
\end{equation*}
With the aid of $\eqref{2.2}_2$, we have
 \begin{equation}\label{2.64}
 |v_x | \leq C(|u_t| + |uu_x| + |n(\omega - u)|) \leq C(|u_t| + |u_x| + |(\omega - u)|), 
\end{equation}
then it holds that
\begin{equation*}  
 |v_x|^{2} \leq C\big(|u_t|^{2} + |u_x|^{2} + |(\omega - u)|^{2}\big), 
 \end{equation*} 
therefore, using \eqref{Sob} and the assumption in Proposition \ref{a-priori} gives
\begin{equation*}
\begin{aligned}
\int_{\mathbb{R}} u_xv_x^{2} dx \leq &\|u_x\|_{\infty} \|u_t\|^{2} + \|u_x\|_{\infty} \| \|u_x\|^{2} + \|u_x\|_{\infty}\|\omega -u\|^{2}\\
\leq&C {\epsilon}^{\frac{1}{2}}{\mathcal{E}}^{\frac{1}{2}} (\|u_t\|^{2} + \|u_x\|^{2} + \|\omega - u\|^{2}).
\end{aligned}
\end{equation*}
Similarly, by \eqref{Sob} and Young's inequality, it follows that
\begin{equation*}
\begin{aligned} 
\int_{\mathbb{R}} n_x(\omega - u)u_x dx \leq \delta \|u_x\|^{2} + C_\delta {\epsilon}\|n_x\|^{2},
\end{aligned}
\end{equation*}
and 
\begin{equation*}
\begin{aligned}
\int_{\mathbb{R}} n\omega_xu_x dx  \leq \delta \|u_x\|^{2} + C_\delta \|\omega_x\|^{2}.
\end{aligned}
\end{equation*}
Plugging the above estimates into \eqref{2.62} and choosing $\delta$ suitably small yield
\begin{equation}\label{2.069}
\begin{aligned}
&  \frac{d}{dt} \int_{\mathbb{R}}\frac{1}{2}(v_x^{2} + u_x^{2}) dx + \int_{\mathbb{R}} nu_x^{2} dx  \\
\leq &C{\epsilon}^{\frac{1}{2}}{\mathcal{E}}^{\frac{1}{2}} (\|u_x\|^{2} + \|u_t\|^{2}) + C(\|\omega_x\|^{2} + \|n_x\|^{2} + \|\omega - u\|^{2}).
\end{aligned}
\end{equation}
To estimate $ \|u_t\|^{2} $, differentiating $\eqref{2.2}_1, \eqref{2.2}_2 $ with respect to the time variable $t$, and multiply the resulted equations by $ v_t, u_t $ respectively, summing them together and integrating over $\mathbb{R} \times (0,t),$ give
\begin{equation}\label{(2.69)}
\begin{aligned}
&\frac{d}{dt} \int_{\mathbb{R}}\frac{1}{2}(v_t^{2} + u_t^{2}) dx + \int_{\mathbb{R}} nu_t^{2} dx \\
=& -\frac{1}{2}\int_{\mathbb{R}} u_xu_t^{2} dx -\frac{\gamma}{2} \int_{\mathbb{R}} u_xv_t^{2} dx - \int_{\mathbb{R}} u_tv_tv_x dx \\
& + \int_{\mathbb{R}} n_t(\omega - u)u_t dx + \int_{\mathbb{R}} n\omega_tu_t dx.
\end{aligned}
\end{equation} 
By using \eqref{Sob} and a priori assumption, we can get 
\begin{equation*}
\int_{\mathbb{R}} u_xu_t^{2} dx \leq C\|u_x\|_{\infty}\|u_t\|^{2} \leq C\|u_x\|^{\frac{1}{2}}\|u_{xx}\|^{\frac{1}{2}}\|u_t\|^{2} \leq C{\epsilon}^{\frac{1}{2}}{\mathcal{E}}^{\frac{1}{2}}\|u_t\|^{2}.
\end{equation*}
By $\eqref{2.2}_1$, we have 
\begin{equation}\label{2.70}
|v_t| \leq C(|u_x| + |uv_x|) \leq  C(|u_x| + |v_x|),
\end{equation}
thus,
\begin{equation*}
|v_t|^{2} \leq  C(|u_x|^{2} + |v_x|^{2}),
\end{equation*}
therefore,
\begin{equation*}
\int_{\mathbb{R}} u_xv_t^{2} dx  \leq C {\epsilon}^{\frac{1}{2}}{\mathcal{E}}^{\frac{1}{2}} (\|u_t\|^{2} + \|u_x\|^{2}).
\end{equation*}
Similarly, by \eqref{2.64}, \eqref{2.70} and Young's inequality, we get
\begin{equation*}
\begin{aligned}
& \int_{\mathbb{R}} u_tv_tv_x dx    \\
\leq&C{\epsilon}^{\frac{1}{2}}{\mathcal{E}}^{\frac{1}{2}}\|u_t\|^{2} + \delta\|u_t\|^{2} + C{\epsilon} {\mathcal{E}} \|u_x\|^{2} + C{\epsilon} {\mathcal{E}}\|\omega-u\|^{2}.  
\end{aligned}
\end{equation*}
Noting that $ n_t = -(n\omega)_x = -(n_x\omega + n\omega_x) $, we have
\begin{equation*}
\begin{aligned}
&  \int_{\mathbb{R}} n_t(\omega - u)u_t dx = - \int_{\mathbb{R}} (n_x\omega + n\omega_x) (\omega - u)u_t dx\\
 =& - \int_{\mathbb{R}} [(n_x\omega(\omega - u)u_t + n\omega_x(\omega - u)u_t] dx\\
\leq&\delta \|u_t\|^{2} + C{\epsilon}{\mathcal{E}}\|\omega -u\|^{2}.
\end{aligned}
\end{equation*}
Similarly, noting that
\begin{equation*}
n\omega_t = -n\omega\omega_x - n_x - n_x\omega_x + n\omega_{xx} + \rho n(u - \omega),
\end{equation*}
thus, 
\begin{equation*}
\int_{\mathbb{R}} n\omega_tu_t dx \leq \delta \|u_t\|^{2} + C\|(\omega_x,n_x,\omega_{xx},u-\omega)\|^{2}.
\end{equation*}
Plugging the above estimates into \eqref{(2.69)} and choosing $\delta $ suitably small, we can obtain
\begin{equation}\label{2.83}
\begin{aligned}
&\frac{d}{dt} \int_{\mathbb{R}} \frac{1}{2}(v_t^{2} + u_t^{2}) dx + \int_{\mathbb{R}} nu_t^{2} dx  \\
 \leq &C {\epsilon}^{\frac{1}{2}}{\mathcal{E}}^{\frac{1}{2}} \|(u_t,u_x)\|^{2} + C\|(\omega_x,n_x,\omega_{xx},u-\omega)\|^{2}.
\end{aligned}
\end{equation}
Combining\eqref{2.069} and \eqref{2.83}, then integrating with respect to $t$ over $[0,t]$ and using the assumption $ {\epsilon}{\mathcal{E}}\leq \frac{1}{16}$ yield
\begin{equation}\label{est_vxux}
\begin{aligned}
&\|(v_x, u_x, v_{t},u_{t})(t,\cdot)\|^{2} + \int_{0}^{t} \|(u_x,u_{\tau})\|^{2}d\tau \\
\leq & C\|(v_{0x}, u_{0x})\|^{2} + \int_{\mathbb{R}} (v^{2}_t+ u^{2}_t)dx\Big|_{t=0} \\
\leq&C\|(v_{0x}, u_{0x}, u_0, \omega_0)\|^{2}  +  C\int_{0}^{t}\|(\omega_x, n_x, \omega_{xx}, u-\omega)\|^{2}d\tau,
\end{aligned}
\end{equation}
where we have used the fact $\int_{\mathbb{R}} (v^{2}_t+ u^{2}_t)dx\Big|_{t=0} \leq C\|(v_{0x}, u_{0x}, u_0, \omega_0)\|^{2}$ which can be derived by using equations $\eqref{2.2}_1, \eqref{2.2}_2 $ in the last inequality. 

Combing \eqref{est_vxux} and Lemmas \ref{thm 2.1}-\ref{thm 2.3}, we complete the proof of Lemma \ref{thm 2.4}.
\end{proof}

\subsection{Estimate of $ n_{xx}, \omega_{xx} $ }
\begin{lma}\label{thm 2.5}
Under the hypothesis and a priori assumptions in Proposition \ref{a-priori}, there exists $C>0$ such that the following estimate holds:
\begin{equation*}
\|(n_{xx},\omega_{xx})(t,\cdot)\|^{2} + \int_{0}^{t} \|(n_{xx}, \omega_{xxx})(\tau,\cdot)\|^{2}d\tau \leq C(\epsilon_{0}^{2} +{\mathcal{E}^2_0}).
\end{equation*}
\end{lma}
\begin{proof}
Differentiating $\eqref{2.2}_3$ twice with respect to $x$ and then multiplying the resulted equation by $n_{xx}$ yield
\begin{equation}\label{2.88}
\begin{aligned}
\frac{1}{2n}(n_{xx}^{2})_t + \frac{1}{n}n_{xx}n_{xxx}\omega + \frac{3}{n}n_{xx}^{2}\omega_x + \frac{3}{n}n_xn_{xx}\omega_{xx} + n_{xx}\omega_{xxx} = 0.
\end{aligned}
\end{equation}
Dividing\eqref{2.012} by $n$, differentiating the result with respect to $x$, then multiplying it by $ n_{xx} $, one gets
\begin{equation}\label{2.89}
\begin{aligned}
& \omega_{xt}n_{xx} + \omega_x^{2}n_{xx} + \omega\omega_{xx}n_{xx} + \frac{1}{n}n_{xx}^{2} - \frac{1}{n}n_x^{2}n_{xx} = \frac{1}{n}\omega_xn_{xx}^{2} + \omega_{xxx}n_{xx} \\
& - \frac{1}{n^{2}}\omega_xn_x^{2}n_{xx} + \rho_x(u-\omega)n_{xx} + \rho (u-\omega)_xn_{xx}.
\end{aligned}
\end{equation}
Adding \eqref{2.88} with \eqref{2.89} together, and integrating by parts give
\begin{equation}\label{2.90}
\begin{aligned}
& (\frac{1}{2n}n_{xx}^{2})_t + \frac{1}{n}n_{xx}^{2} = - \frac{1}{n}\omega_xn_{xx}^{2} - \frac{3}{n}n_xn_{xx}\omega_{xx} - (\omega_xn_{xx})_t + (\omega_xn_{xt})_x  \\
&+ \omega_{xx}n_{xt} - \omega_x^{2}n_{xx} - \omega\omega_{xx}n_{xx} +\frac{1}{n}n_x^{2}n_{xx} - \frac{1}{n^{2}}\omega_xn_x^{2}n_{xx} \\
& + \rho_x(u-\omega)n_{xx}  + \rho (u-\omega)_xn_{xx}.
\end{aligned}
\end{equation}
Integrating the above equation with respect to $x$ and $t$ over $ \mathbb{R}\times [0, t]$, we get 
\begin{equation}\label{2.92}
\begin{aligned}
& \int_{\mathbb{R}}\frac{1}{2n}n_{xx}^{2} dx + \int_{0}^{t}\int_{\mathbb{R}}\frac{1}{n}n_{xx}^{2} dxd\tau = \int_{\mathbb{R}}\frac{1}{2n_0}n_{0xx}^{2} dx + \int_{0}^{t}\int_{\mathbb{R}} \bigg[-\frac{1}{n}\omega_xn_{xx}^{2}  \\
&-\frac{3}{n}n_xn_{xx}\omega_{xx} + \omega_{xx}n_{xt} - \omega_x^{2}n_{xx} - \omega\omega_{xx}n_{xx}  + \frac{1}{n}n_x^{2}n_{xx}  - \frac{1}{n^{2}}\omega_xn_x^{2}n_{xx}  \\
& + \rho_x(u-\omega)n_{xx}   + \rho (u-\omega)_xn_{xx} \bigg]dxd\tau + \int_{\mathbb{R}}\omega_xn_{xx} dx + \int_{\mathbb{R}}\omega_{0x}n_{0xx} dx.
\end{aligned}
\end{equation}
By \eqref{Sob} and Lemma \ref{thm 2.3}, we have 
\begin{equation*}
\begin{aligned}
& \int_{0}^{t}\int_{\mathbb{R}} \frac{1}{n}\omega_xn_{xx}^{2} dxd\tau  \leq C \int_{0}^{t} \|\omega_x\|_{\infty}\|n_{xx}\|^{2} d\tau   \\
\leq & C  \int_{0}^{t}\|\omega_x\|^{\frac{1}{2}}\|\omega_{xx}\|^{\frac{1}{2}}\|n_{xx}\|^{2} d\tau \leq C{\epsilon}^{\frac{1}{2}} {\mathcal{E}}^{\frac{1}{2}} \int_{0}^{t}\|n_{xx}\|^{2} d\tau.
\end{aligned}
\end{equation*}
Similarly,
\begin{equation*}
\begin{aligned}
& \int_{0}^{t}\int_{\mathbb{R}}\bigg( \frac{3}{n}n_xn_{xx}\omega_{xx} +\omega_x^{2}n_{xx} +  \omega\omega_{xx}n_{xx} + \frac{1}{n}n_x^{2}n_{xx} +  \frac{1}{n^{2}}\omega_xn_x^{2}n_{xx}   \\
& \qquad \qquad + \rho_x(u-\omega)n_{xx}  +  \rho (u-\omega)_xn_{xx}  \bigg)dxd\tau   \\
\leq & C\left(\epsilon^2 + \delta\right) \int_{0}^{t}\|n_{xx}\|^{2} d\tau + C {\epsilon}{\mathcal{E}} \int_{0}^{t}\|(\omega_x, \omega_{xx}, u-\omega)\|^{2} d\tau   \\
& + C{\epsilon}^{4} \int_{0}^{t}\|n_x\|^{2} d\tau  + C\int_{0}^{t} \|(u_x,\omega_x)\|^{2}d\tau,  \\
\end{aligned}
\end{equation*}
the last two terms in \eqref{2.92}  can be estimated by
\begin{equation*}
\begin{aligned}
& \int_{\mathbb{R}}\omega_xn_{xx} dx  \leq \delta \|n_{xx}\|^{2} + C\|\omega_x\|^{2}, \\
& \int_{\mathbb{R}}\omega_{0x}n_{0xx} dx \leq C(\|\omega_{0x}\|^{2} + \|n_{0xx}\|^{2}).
\end{aligned}
\end{equation*}
Finally we estimate the term  $\di\int_{0}^{t}\int_{\mathbb{R}} \omega_{xx}n_{x\tau} dxd\tau$. By virtue of $ \eqref{2.2}_3 $, we have
$$ n_{xt} = -(n\omega)_{xx} = - (n_{xx}\omega + 2n_x\omega_x + n\omega_{xx}).$$ 
It is easy to see that
\begin{equation*}
\begin{aligned}
\int_{0}^{t}\int_{\mathbb{R}} \omega_{xx}n_{x\tau} dxd\tau  \leq \delta \int_{0}^{t}\|n_{xx}\|^{2} d\tau + C{\epsilon}\int_{0}^{t} \|\omega_x\|^{2} d\tau + C\int_{0}^{t} \|\omega_{xx}\|^{2} d\tau.
\end{aligned}
\end{equation*}
Plugging the above estimates into \eqref{2.92}, by the smallness of $\delta, \epsilon$ and the assumption $\epsilon \mathcal{E} \leq \frac{1}{16}$, we have 
\begin{equation}\label{2.108}
\begin{aligned}
& \|n_{xx}\|^{2} + \int_{0}^{t}\|n_{xx}\|^{2} d\tau \leq  C\|(\omega_{0x},n_{0xx})\|^{2}+  C\int_{0}^{t} \|(u_x,\omega_x,\omega_{xx})\|^{2} d\tau   \\
&  + C{\epsilon}^{4}\int_{0}^{t}\|n_x\|^{2} d\tau +  C{\epsilon}{\mathcal{E}}\int_{0}^{t}\|u-\omega\|^{2} d\tau + C\|\omega_x\|^{2}. 
\end{aligned}
\end{equation}
Dividing \eqref{2.012} by $n$, then differentiating the resulted equation with respect to $x$, and multiplying it by $ -n\omega_{xxx} $, we get
\begin{equation}\label{2.112}
\begin{aligned}
& \frac{1}{2}n(\omega_{xx})_t^{2} + n\omega_{xxx}^{2} =
n(\omega_{xt}\omega_{xx})_x + n\omega_x^{2}\omega_{xxx} + n\omega\omega_{xx}\omega_{xxx} \\
& + n_{xx}\omega_{xxx}  - \frac{1}{n}n_x^{2}\omega_{xxx}
-\omega_xn_{xx}\omega_{xxx}  + \frac{1}{n}\omega_xn_x^{2}\omega_{xxx} \\
& -n\rho_x(u-\omega)\omega_{xxx}  - n\rho (u-\omega)_x\omega_{xxx}.
\end{aligned}
\end{equation}
Combining \eqref{2.88} and \eqref{2.112}, a direct computation implies
\begin{equation}\label{2.047}
\begin{aligned}
& (\frac{1}{2n}n_{xx}^{2})_t - \frac{1}{2}n_{xx}^{2}(\frac{1}{n})_t + (\frac{1}{2}n\omega_{xx}^{2})_t - \frac{1}{2}n_t\omega_{xx}^{2} + n\omega_{xxx}^{2} \\
=  &- \frac{1}{n}n_{xx}n_{xxx}\omega  - \frac{3}{n}n_{xx}^{2}\omega_x  - \frac{3}{n}n_xn_{xx}\omega_{xx} + 
n(\omega_{xt}\omega_{xx})_x + n\omega_x^{2}\omega_{xxx}  \\
& + n\omega\omega_{xx}\omega_{xxx} - \frac{1}{n}n_x^{2}\omega_{xxx}
-\omega_xn_{xx}\omega_{xxx}  + \frac{1}{n}\omega_xn_x^{2}\omega_{xxx}-n\rho_x(u-\omega)\omega_{xxx} \\
& - n\rho (u-\omega)_x\omega_{xxx}.
\end{aligned}
\end{equation}
Integrating \eqref{2.047} with respect to $x$ and $t$ over $\mathbb{R}\times [0,t]$ yields 
\begin{equation}\label{2.115}
\begin{aligned}
& \int_{\mathbb{R}} (\frac{1}{2n}n_{xx}^{2} + \frac{1}{2}n\omega_{xx}^{2}) dx + \int_{0}^{t}\int_{\mathbb{R}} n\omega_{xxx}^{2} dxd\tau \\
=& \int_{\mathbb{R}} (\frac{1}{2n_0}n_{0xx}^{2} + \frac{1}{2}n_0\omega_{0xx}^{2}) dx   +\int_{0}^{t}\int_{\mathbb{R}}\left[ -\frac{2}{n}\omega_xn_{xx}^{2}  -  \frac{1}{2}n_x\omega\omega_{xx}^{2} 
- \frac{1}{2}n\omega_x\omega_{xx}^{2} \right.\\
& - \frac{3}{n}n_xn_{xx}\omega_{xx} +  n\omega_x^{2}\omega_{xxx} 
 +  n\omega\omega_{xx}\omega_{xxx} -  \frac{1}{n}n_x^{2}\omega_{xxx} -  \omega_xn_{xx}\omega_{xxx} \\
 & \left.+  \frac{1}{n}\omega_xn_x^{2}\omega_{xxx} -  n\rho_x(u-\omega)\omega_{xxx}   - n\rho (u-\omega)_x\omega_{xxx}\right] dxd\tau.
\end{aligned}
\end{equation}
Using Lemmas \ref{1.1} and \ref{thm 2.3}, we have 
\begin{equation*}
\begin{aligned}
&\int_{0}^{t}\int_{\mathbb{R}} \frac{2}{n}\omega_xn_{xx}^{2} dxd\tau \leq C\int_{0}^{t}\|\omega_x\|_{\infty}\|n_{xx}\|^{2} d\tau \\
\leq &C\int_{0}^{t} \|\omega_x\|^{\frac{1}{2}}\|\omega_{xx}\|^{\frac{1}{2}}\|n_{xx}\|^{2} d\tau \leq C{\epsilon}^{\frac{1}{2}}{\mathcal{E}}^{\frac{1}{2}}\int_{0}^{t}\|n_{xx}\|^{2} d\tau.
\end{aligned}
\end{equation*}
Similarly, 
\begin{equation*}
\begin{aligned}
& \int_{0}^{t}\int_{\mathbb{R}}\bigg( \frac{1}{2}n_x\omega\omega_{xx}^{2} + \frac{1}{2}n\omega_x\omega_{xx}^{2} +  n\omega_{x}^{2}\omega_{xxx} + \frac{3}{n}n_xn_{xx}\omega_{xx} + n\omega\omega_{xx}\omega_{xxx}   +\frac{1}{n}n_x^{2}\omega_{xxx}    \\
& +  \omega_xn_{xx}\omega_{xxx} +  \frac{1}{n}\omega_xn_x^{2}\omega_{xxx} + n\rho_x(u-\omega)\omega_{xxx} +  n\rho (u-\omega)_x\omega_{xxx}  \bigg)dxd\tau   \\
\leq &  \delta \int_{0}^{t} \|\omega_{xxx}\|^{2} d\tau + \delta \int_{0}^{t}\|n_{xx}\|^{2} d\tau + C{\epsilon}^{2}\int_{0}^{t}\|(\omega_x, n_x, n_{xx})\|^{2} d\tau,  \\
&  + C {\epsilon}{\mathcal{E}}\int_{0}^{t}\|(n_{xx}, u-\omega)\|^{2} d\tau  + C\int_{0}^{t} \|(u_x,\omega_x,\omega_{xx})\|^{2} d\tau.
\end{aligned}
\end{equation*}
Plugging the above estimates into \eqref{2.115} and choosing $\delta$ suitably small, we have 
\begin{equation}\label{2.043}
\begin{aligned}
& \|(n_{xx},\omega_{xx})\|^{2} + \int_{0}^{t} \|\omega_{xxx}\|^{2} d\tau   \\
\leq&C\|(n_{0xx},\omega_{0xx})\|^{2}   
+ C\int_{0}^{t}\|(u_x,\omega_x, \omega_{xx})\|^{2} d\tau + C{\epsilon}^{2} \int_{0}^{t}\|n_x\|^{2} d\tau \\
& + C{\epsilon}^{\frac{1}{2}}{\mathcal{E}}^{\frac{1}{2}} \int_{0}^{t} \|n_{xx}\|^{2} d\tau + C{\epsilon}{\mathcal{E}} \int_{0}^{t} \|u-\omega\|^{2} d\tau.
\end{aligned}
\end{equation}
By \eqref{2.108}, \eqref{2.043}, Lemmas \ref{thm 2.1}-\ref{thm 2.4}, we can complete the proof of Lemma \ref{thm 2.5}.
\end{proof}

\subsection{Estimate of $ v_{xx}, u_{xx}, v_{xt}, u_{xt} $ }
\begin{lma}\label{thm 2.6}
Under the hypothesis and a priori assumptions in Proposition \ref{a-priori}, there exists $C>0$ such that the following estimate holds:
\begin{equation}
\begin{aligned}
& \|(v_{xx}, u_{xx},v_{xt},u_{xt})(t,\cdot)\|^{2} + \int_{0}^{t} \|(u_{xx},u_{x\tau})(\tau,\cdot)\|^{2} d\tau  \\
\leq&C(\epsilon_{0}^{2}+{\mathcal{E}^2_0}).
\end{aligned}
\end{equation}
\end{lma}

\begin{proof}
Differentiating $\eqref{2.2}_1, \eqref{2.2}_2$ twice with respect to $x$, multiplying two resulted equations by $ v_{xx}, u_{xx} $ respectively, and then summing them together and integrating over $\mathbb{R}\times[0,t]$, we have
\begin{equation}\label{2.127}
\begin{aligned}
& \frac{1}{2}\int_{\mathbb{R}} (v_{xx}^{2} + u_{xx}^{2}) dx + \int_{0}^{t}\int_{\mathbb{R}} nu_{xx}^{2} dx \\
 =&  \frac{1}{2}\int_{\mathbb{R}} (u_{0xx}^{2} + v_{0xx}^{2}) dx +  \int_{0}^{t}\int_{\mathbb{R}} \big[-\frac{5}{2} u_xu_{xx}^{2} 
- \frac{1}{2}(\gamma + 2)u_xv_{xx}^{2}   \\
& - (2\gamma - 1) v_xu_{xx}v_{xx} + n_{xx}(\omega - u)u_{xx} + 2n_{x}(\omega - u)_xu_{xx} + n\omega_{xx}u_{xx}\big] dxd\tau.
\end{aligned}
\end{equation}
Using \eqref{Sob} and Lemma \ref{thm 2.4} leads to 
\begin{equation*}
\begin{aligned}
& \int_{0}^{t}\int_{\mathbb{R}} u_xu_{xx}^{2} dxd\tau   \\
\leq &C\int_{0}^{t} \|u_x\|_{\infty}\|u_{xx}\|^{2} d\tau \leq C\int_{0}^{t} \|u_x\|^{\frac{1}{2}}\|u_{xx}\|^{\frac{1}{2}}\|u_{xx}\|^{2} d\tau \\
\leq&C{\epsilon}^{\frac{1}{2}}{\mathcal{E}}^{\frac{1}{2}}\int_{0}^{t}\|u_{xx}\|^{2} d\tau
\end{aligned}
\end{equation*}
Noting from $\eqref{2.2}_2$ that
\begin{equation}\label{2.046}
 |v_{xx}| \leq C(|u_{xt}| + |u_{xx}| + |n_x(\omega - u)| + |n(\omega - u)_x|),
\end{equation}
thus,
\begin{equation*}
\begin{aligned}
& \int_{0}^{t}\int_{\mathbb{R}} u_xv_{xx}^{2} dxd\tau   \\
\leq &{\epsilon}^{\frac{1}{2}}{\mathcal{E}}^{\frac{1}{2}}\int_{0}^{t} \|(u_{x\tau},u_{xx})\|^{2} d\tau 
+  C{\epsilon}^{\frac{3}{2}}{\mathcal{E}}^{\frac{3}{2}}\int_{0}^{t} \|\omega - u\|^{2} d\tau + C{\epsilon}{\mathcal{E}}\int_{0}^{t} \|(\omega_x,u_x)\|^{2} d\tau.  
\end{aligned}
\end{equation*}
Similarly, using \eqref{2.046} and Young's inequality yields
\begin{equation*}\label{2.140}
\begin{aligned}
& \int_{0}^{t}\int_{\mathbb{R}} v_xu_{xx}v_{xx} dxd\tau    \\
\leq&\delta \int_{0}^{t}\|u_{xx}\|^{2} d\tau + C{\epsilon}{\mathcal{E}} \int_{0}^{t} \|u_{x\tau}\|^{2} d\tau + C{\epsilon}^{\frac{1}{2}}{\mathcal{E}}^{\frac{1}{2}} \int_{0}^{t} |\|u_{xx}\|^{2} d\tau   \\
& + C{\epsilon}^{2}{\mathcal{E}}^{2} \int_{0}^{t} \|\omega - u\|^{2} d\tau + C{\epsilon}{\mathcal{E}} \int_{0}^{t} \|(\omega_x,u_x)\|^{2} d\tau. 
\end{aligned}
\end{equation*}
Applying \eqref{Sob} and Young's inequality, we have
\begin{equation*}
\begin{aligned}
& \int_{0}^{t}\int_{\mathbb{R}} \left(n_{xx}(\omega - u)u_{xx} +  n_{x}(\omega - u)_xu_{xx} +n\omega_{xx}u_{xx}  \right)dxd\tau   \\
 \leq & \delta \int_{0}^{t}\|u_{xx}\|^{2} d\tau + C{\epsilon}^{2} \int_{0}^{t} \|n_{xx}\|^{2} d\tau  + C{\epsilon}{\mathcal{E}} \int_{0}^{t} \|(\omega_x,u_x)\|^{2} d\tau + C\int_{0}^{t} \|\omega_{xx}\|^{2} d\tau.
\end{aligned}
\end{equation*}
Plugging the above estimates into \eqref{2.127} and taking $\delta$ suitably small yield 
\begin{equation}\label{2.144}
\begin{aligned}
 & \|(v_{xx}, u_{xx})\|^{2} + \int_{0}^{t}\|u_{xx}\|^{2} d\tau   \\
\leq&\|(v_{0xx},u_{0xx})\|^{2}) + C{\epsilon}{\mathcal{E}}\int_{0}^{t} \|(\omega_x,u_x)\|^{2} d\tau  + C{\epsilon}^{\frac{1}{2}}{\mathcal{E}}^{\frac{1}{2}}\int_{0}^{t}\|(u_{xx},u_{x\tau})\|^{2} d\tau   \\
& + C{\epsilon}^{2} \int_{0}^{t} \|n_{xx}\|^{2} d\tau   + C\int_{0}^{t} \|\omega_{xx}\|^{2} d\tau
+  C{\epsilon}^{\frac{3}{2}}{\mathcal{E}}^{\frac{3}{2}}\int_{0}^{t} \|\omega - u\|^{2} d\tau.
\end{aligned}
\end{equation}
On the other hand, to estimate $\int_{0}^{t} \|u_{x\tau}\|^{2} d\tau $, applying $\partial_{xt}$ to $\eqref{2.2}_1$ and $\eqref{2.2}_2$ respectively, and then multiplying the resulted equations by $ v_{xt}, u_{xt} $ respectively, summing them together and then integrating over $\mathbb{R}\times [0,t]$, we can get
\begin{equation}\label{234}
\begin{aligned}
& \frac{1}{2}\int_{\mathbb{R}} (v_{xt}^{2} + u_{xt}^{2}) dx + \int_{0}^{t}\int_{\mathbb{R}} nu_{x\tau}^{2} dxd\tau \\
=&  \frac{1}{2}\int_{\mathbb{R}} (v_{xt}^{2} + u_{xt}^{2}) dx\Big|_{t=0}  +\int_{0}^{t}\int_{\mathbb{R}}\bigg[ -\frac{3}{2}u_xu_{x\tau}^{2} - \frac{\gamma}{2} u_xv_{x\tau}^{2} - \gamma v_xu_{x\tau}v_{x\tau}
\\
&- u_\tau v_{xx}v_{x\tau} -  u_{\tau}u_{xx}u_{x\tau} - \frac{\gamma - 1}{2} v_\tau u_{xx}v_{x\tau}- \frac{\gamma - 1}{2} v_{\tau}v_{xx}v_{x\tau} +  n_{x\tau}(\omega - u )u_{x\tau} \\
&+ n_x(\omega - u)_\tau u_{x\tau}+ n_\tau(\omega-u)_xu_{x\tau} + n\omega_{x\tau}u_{x\tau}\bigg]dxd\tau.
\end{aligned}
\end{equation}

By \eqref{Sob} and Lemma \ref{thm 2.4}, one can get
\begin{equation*}
\begin{aligned}
& \int_{0}^{t}\int_{\mathbb{R}} u_xu_{x\tau}^{2} dxd\tau  \\
\leq &C\int_{0}^{t} \|u_x\|_{\infty}\|u_{x\tau}\|^{2} d\tau  
\leq C\int_{0}^{t} \|u_x\|^{\frac{1}{2}}\|u_{xx}\|^{\frac{1}{2}}\|u_{x\tau}\|^{2} d\tau  \leq  C{\epsilon}^{\frac{1}{2}}{\mathcal{E}}^{\frac{1}{2}}\int_{0}^{t} \|u_{x\tau}\|^{2} d\tau.
\end{aligned}
\end{equation*}
To estimate $ \int_{0}^{t}\int_{\mathbb{R}} u_xv_{x\tau}^{2} dxd\tau $, noting from $\eqref{2.2}_1$ that
 $$ |v_{xt}| \leq  C(|u_{xx}| + |v_{xx}|+|u_xv_x|),$$
hence,
\begin{equation*}
\begin{aligned}
& \int_{0}^{t}\int_{\mathbb{R}} u_xv_{x\tau}^{2} dxd\tau \leq C{\epsilon}^{\frac{1}{2}}{\mathcal{E}}^{\frac{1}{2}}\int_{0}^{t}\|( u_x,u_{xx},v_{xx})\|^{2} d\tau.
\end{aligned}
\end{equation*}
Then we estimate $\int_{0}^{t} \|v_{xx}\|^{2} d\tau$. Using \eqref{2.046} yields  
\begin{equation*}
\begin{aligned}
\int_{0}^{t} \|v_{xx}\|^{2} d\tau \leq C\int_{0}^{t} \|(\omega_x, u_x,u_{x\tau},u_{xx})\|^{2} d\tau + C{\epsilon}^{\frac{1}{2}}{\mathcal{E}}^{\frac{1}{2}}\int_{0}^{t}\|(\omega - u)\|^{2} d\tau,
\end{aligned}
\end{equation*}
thus, we have 
\begin{equation*}
\begin{aligned}
& \int_{0}^{t}\int_{\mathbb{R}} u_xv_{x\tau}^{2} dxd\tau   \\
\leq&C{\epsilon}^{\frac{1}{2}}{\mathcal{E}}^{\frac{1}{2}}\int_{0}^{t}\|(u_{xx}, u_{x\tau}, \omega_x,u_x)\|^{2} d\tau + C{\epsilon}{\mathcal{E}}\int_{0}^{t} \|\omega - u\|^{2} d\tau   \\
\end{aligned}
\end{equation*}
Similarly, we have
\begin{equation*}
\begin{aligned} 
& \int_{0}^{t}\int_{\mathbb{R}} \bigg(v_xu_{x\tau}v_{x\tau}  + u_\tau v_{xx}v_{x\tau} +  u_\tau u_{xx}u_{x\tau} + v_\tau u_{xx}v_{x\tau}  +v_\tau v_{xx}v_{x\tau}  \\
& +  n_{x\tau}(\omega - u )u_{x\tau} \bigg)dxd\tau   \\
\leq &C\delta \left(\int_{0}^{t}\|u_{x\tau}\|^{2} d\tau + \int_{0}^{t}\|u_{xx}\|^{2} d\tau\right) + {\epsilon}^{\frac{1}{2}}{\mathcal{E}}^{\frac{1}{2}}\int_{0}^{t}  \|(\omega - u,\omega_x,u_x)\|^{2} d\tau    \\
& + C {\epsilon}^{\frac{1}{2}}{\mathcal{E}}^{\frac{1}{2}} \int_{0}^{t} \|(u_{x\tau},u_{xx})\|^{2} d\tau  +C{\epsilon}^{4}\int_{0}^{t} \|n_{xx}\|^{2} d\tau  +  C{\epsilon}^{2} \int_{0}^{t} \|\omega_{xx}\|^{2} d\tau.
\end{aligned}
\end{equation*} 
Since,
\begin{equation*}
\begin{aligned} 
\int_{0}^{t}\int_{\mathbb{R}} n_x(\omega - u)_\tau u_{x\tau} dxd\tau = \int_{0}^{t}\int_{\mathbb{R}} n_x\omega_\tau u_{x\tau} dxd\tau - \int_{0}^{t}\int_{\mathbb{R}} n_xu_\tau u_{x\tau} dxd\tau,
\end{aligned}
\end{equation*}
to estimate $ \int_{0}^{t}\int_{\mathbb{R}} n_x\omega_\tau u_{x\tau} dxd\tau $, noting that 
$$\omega_t = -\omega\omega_x - \frac{1}{n}n_x + \frac{1}{n}n_x\omega_x + \omega_{xx} + \rho (u - \omega),$$
we can  conclude that 
\begin{equation*}
\begin{aligned} 
& \int_{0}^{t}\int_{\mathbb{R}} n_x(\omega - u)_\tau u_{x\tau} dxd\tau\\
\leq&\delta\int_{0}^{t} \|u_{x\tau}\|^{2} d\tau + C{\epsilon}^{\frac{1}{2}}{\mathcal{E}}^{\frac{1}{2}} \int_{0}^{t} \|(n_x, \omega_x, \omega_{xx},u - \omega)\|^{2} d\tau + C{\epsilon}^{2}\int_{0}^{t} \|n_{xx}\|^{2} d\tau.
\end{aligned}
\end{equation*}
Similarly, we have
\begin{equation*}
\begin{aligned} 
& \int_{0}^{t}\int_{\mathbb{R}} n_\tau (\omega-u)_xu_{x\tau} dxd\tau \\ 
\leq&\delta\int_{0}^{t} \|u_{x\tau}\|^{2} d\tau + C {\epsilon}{\mathcal{E}}\int_{0}^{t} \|(\omega_x,u_x)\|^{2} d\tau.
\end{aligned}
\end{equation*}
To estimate $  \int_{0}^{t}\int_{\mathbb{R}} n\omega_{x\tau}u_{x\tau} dxd\tau $,
noting that 
$$
\begin{array}{ll}
\di \omega_{xt} = -\omega_{x}^{2} -\omega\omega_{xx} - \frac{1}{n}n_{xx} - 
\frac{1}{n^{2}}n_x^{2} + \frac{1}{n}n_{xx}\omega_x + \omega_{xxx} -\frac{1}{n^{2}}n_x^{2}\omega_x\\
\di\qquad \ \   + \frac{1}{n}n_x\omega_{xx} + \rho_x(u -\omega)n_{xx} + \rho(u-\omega)_xn_{xx}, 
\end{array}$$
we can obtain
\begin{equation*}
\begin{aligned} 
& \int_{0}^{t}\int_{\mathbb{R}} n\omega_{x\tau}u_{x\tau} dxd\tau \\
\leq&\delta \int_{0}^{t} \|u_{x\tau}\|^{2} d\tau + C{\epsilon}{\mathcal{E}} \int_{0}^{t} \|(\omega_x,n_x)\|^{2}d\tau + C\left(\int_{0}^{t} \|\omega_{xx}\|^{2} d\tau + \int_{0}^{t} \|n_{xx}\|^{2} d\tau\right). 
\end{aligned}
\end{equation*}
Plugging these estimates into \eqref{234}, then combining with \eqref{2.144} together and by choosing $ \delta $ suitably small and the fact $\epsilon \mathcal{E} \leq \frac{1}{16}$ yield
\begin{equation}\label{2.187}
\begin{aligned} 
&  \|(v_{xx},u_{xx},v_{xt},u_{xt})\|^{2}  + \int_{0}^{t}\|(u_{xx},u_{x\tau})\|^{2} d\tau  \\
\leq&\|(v_{0xx},u_{0xx})\|^{2}  + \int_{\mathbb{R}} (v_{xt}^{2} + u_{xt}^{2}) dx\Big|_{t=0} +  C {\epsilon}^{\frac{1}{2}}{\mathcal{E}}^{\frac{1}{2}}  \int_{0}^{t} \|(n_x,\omega_x,u_x)\|^{2} d\tau   \\
& + C\int_{0}^{t} \|(n_{xx},\omega_{xx})\|^{2} d\tau  + C{\epsilon}{\mathcal{E}} \int_{0}^{t}  \|\omega - u\|^{2} d\tau.
\end{aligned}
\end{equation}
By virtue of \eqref{2.187}, Lemmas \ref{thm 2.1}-\ref{thm 2.5}, we obtain
\begin{equation}
\begin{aligned} 
&  \|(v_{xx},u_{xx},v_{xt},u_{xt})(t,\cdot)\|^{2}  + \int_{0}^{t}\|(u_{xx},u_{x\tau})(\tau,\cdot)\|^{2}d\tau \leq C(\epsilon_0^{2}+{\mathcal{E}^2_0}),
\end{aligned}
\end{equation}
where we have used the fact $$\int_{\mathbb{R}} (v_{xt}^{2} + u_{xt}^{2}) dx\Big|_{t=0} \leq C\|(v_{0x}, u_{0x}, v_{0xx},u_{0xx})\|^{2} ,$$
 which can be derived by differentiating $\eqref{2.2}_1, \eqref{2.2}_2$ with $x$.

\end{proof}

Based on the local existence in Proposition \ref{thm 1.1} and a priori estimate in 
Proposition \ref{a-priori}, we can prove Theorem \ref{thm 1.2} as follows.

%\subsection{Proof of Theorem \ref{thm 1.2}}
\begin{proof}
Choosing initial data satisfies $\mathcal{E}_0 < \frac{1}{2} \mathcal{E}$ and  then choosing  $\epsilon_0 < \frac{1}{2}\epsilon$, %Choosing initial data satisfies $ {\mathcal{E}_0} = \frac{1}{4\sqrt{C_0}}{\mathcal{E}}$,
%and then choosing $\epsilon_0 =\min\{ \frac{1}{2\sqrt{C_0}}{\epsilon}, \mathcal{E}_0\}$, 
%$\epsilon_0 = \min\{\frac{1}{2}{\epsilon}, \frac{1}{2C_0}{\epsilon}\}$, $ {\mathcal{E}_0} = \min\{\frac{1}{2}{\mathcal{E}}, \frac{1}{2\sqrt{C_0}}{\mathcal{E}}-1\}$ and satisfying $\epsilon_0\mathcal{E}_0\leq \min\{\frac{1}{64}, \frac{1}{16\sqrt{C}}-\frac{1}{2}\epsilon, \frac{1}{64C}, \frac{1}{64C^{\frac{3}{2}}}-\frac{1}{2C}\epsilon\}$ which leads to $\epsilon \mathcal{E} \leq \frac{1}{16}$, 
by Proposition \ref{thm 1.1}, the local solution of \eqref{1.1}-\eqref{1.2} (or \eqref{2.2}-\eqref{2.3}) exists in $C\left([0,T_0];H^{2}(\mathbb{R})\right)$ and has the estimate
\begin{equation}\label{2.85}
\begin{aligned} 
& \sup\limits_{t\in [0,T_0]}\|(\rho- \rho_{*}, u, n - n_{*}, \omega)(t,\cdot)\|_{H^{1}} \leq 2\epsilon_0 \leq {\epsilon},   \\
& \sup\limits_{t\in [0,T_0]}\|(\rho_{xx}, u_{xx}, \rho_{xt},u_{xt}, n_{xx}, \omega_{xx})(t,\cdot)\| \leq 2{\mathcal{E}_0} \leq {\mathcal{E}},
\end{aligned}
\end{equation}
therefore by Proposition \ref{a-priori}, the solution satisfies a priori estimates 
$$\sup\limits_{t\in [0,T_0]}\|(\rho- \rho_{*}, u, n - n_{*}, \omega)(t,\cdot)\|_{H^{1}} \leq \sqrt{C_0}\epsilon_0 \leq \frac{1}{2}{\epsilon} $$  and 
$$\sup\limits_{t\in [0,T_0]}\|(\rho_{xx}, u_{xx}, \rho_{xt}, u_{xt}, n_{xx}, \omega_{xx})(t,\cdot)\| \leq \sqrt{C_0}(\epsilon_0 + {\mathcal{E}_0})\leq \frac{1}{2}{\mathcal{E}}, $$ 
provided $\epsilon_0 < \frac{1}{2C_0}{\epsilon}$, $ {\mathcal{E}_0} < \frac{1}{4\sqrt{C_0}}{\mathcal{E}} $ and $\epsilon_0 < \mathcal{E}_0$. 
Thus by Proposition \ref{thm 1.1} the initial value problem \eqref{1.1}(or \eqref{2.2}) for $t \geq T_0$ with the initial data $(\rho, u, n, \omega)(T_0)$ has again a unique solution $(\rho, u, n, \omega) \in C([T_0,2T_0];H^{2})$ satisfying the estimates
\begin{equation}\label{2.86}
\begin{aligned} 
& \sup\limits_{t\in [T_0,2T_0]}\|(\rho - \rho_{*}, u, n- n_{*}, \omega)(t,\cdot)\|_{H^{1}} \leq {\epsilon},   \\
& \sup\limits_{t\in [T_0,2T_0]}\|(\rho_{xx}, u_{xx}, \rho_{xt},u_{xt}, n_{xx}, \omega_{xx})(t,\cdot)\| \leq {\mathcal{E}}.
\end{aligned}
\end{equation}
Then by \eqref{2.85}, \eqref{2.86} and Proposition \ref{a-priori}, we have 
$$\sup\limits_{t\in [0,2T_0]}\|(\rho - \rho_{*}, u, n - n_{*}, \omega)(t,\cdot)\|_{H^{1}} \leq \sqrt{C_0}\epsilon_0 \leq \frac{1}{2}{\epsilon} $$  and 
$$\sup\limits_{t\in [0,2T_0]}\|(\rho_{xx}, u_{xx}, \rho_{xt},u_{xt}, n_{xx}, \omega_{xx})(t,\cdot)\|  \leq \sqrt{C_0}(\epsilon_0 + {\mathcal{E}_0}) \leq \frac{1}{2}{\mathcal{E}} ,$$ 
provided $\epsilon_0 < \frac{1}{2C_0}{\epsilon}$, $ {\mathcal{E}_0} < \frac{1}{4\sqrt{C_0}}{\mathcal{E}} $ and $\epsilon_0 < \mathcal{E}_0$. 
%provided $\epsilon_0 = \frac{1}{2C_0}{\epsilon}$ , $ {\mathcal{E}_0} = \frac{1}{4\sqrt{C_0}}{\mathcal{E}}$. 

Therefore we can continue the same process for $0 \leq t\leq nT_0$, $n=3, 4, 5, \cdots$, and finally obtain a global solution $(\rho, u, n, \omega) \in X(0, +\infty)$ satisfying \eqref{11111}.
%\begin{equation}
%\begin{aligned}
%& \sup\limits_{t\in [0,\infty)}\big[\|(\rho -\rho_*, u, n-n_*, \omega)(t,\cdot)\|_{H^{1}}^{2} +\|(\rho_t, u_t)(t,\cdot)\|^{2}\big]  \\
%& + \int_{0}^{\infty}\big[\|(\rho_x, u_x, n_x, \omega_x)\|^{2} +\|(u_t, \omega_{xx})\|^{2} + \|u-\omega\|^{2}\big] dt  \leq C\epsilon_0^{2},  \\
%& \sup\limits_{t\in [0,\infty]}\|(\rho_{xx}, u_{xx},n_{xx}, \omega_{xx}, \rho_{xt}, u_{xt})(t,\cdot)\|^{2} \\
%& + \int_{0}^{\infty} (\|(\rho_{xx}, u_{xx}, n_{xx}, u_{xt})\|^{2} + \|\omega_{xxx}\|^{2} )dt \leq C(\epsilon_{0}^{2}+{\mathcal{E}^2_0}).
%\end{aligned}
%\end{equation}
\end{proof}

\section{Time decay estimates}\label{section 3}
The two parts of this section are necessary to prove the optimal time decay rates.
\subsection{Spectral analysis}
%This section is devoted to the main stage for the proof of Theorem \ref{thm 1.3}. 
In this subsection, we carry out a detailed spectral analysis for the linearized system, based on which we derive a time decay estimate for the linearized equations. 

%To get a refined estimate that is a higher decay rate which play a crucial role in the proof of optimal decay rate for the nonlinear terms, as we do in \eqref{4.1}, we can write the nonlinear term G into two parts which consists of the conserved form part and the non-conserved one that is an essential difficulty. In order to cancel out the non-conserved part, we define $ m := \rho u, M := n\omega $ so that we can make full use of some symmetric form, i.e., the relaxation drag force term have the opposite sign, of the system. Thanks to the upper and lower bound of $\rho$ and n, m and M are equivalent to u and $\omega$ respectively. We linearized the system \eqref{1.1} around the equilibrium state $(\rho_*, 0, n_*,0)$ and get

In order to establish a more refined estimate, which encompasses a higher decay rate that plays a crucial role in proving the optimal decay rate for the nonlinear terms, as demonstrated in equation \eqref{4.1}, we decompose the nonlinear term $G$ into two distinct components, one of which has a conserved-form and the other does not, it is precisely the non-conserved form terms present the essential difficulty. To address this challenge, we introduce the variables $ m:= \rho u $ and $ M:= n \omega $, thereby enabling the utilization of certain symmetric properties to cancel out the non-conserved terms. Specifically, the relaxation drag force terms exhibit an opposite sign in the system. Owing to the upper and lower bounds of $\rho$ and $n$, the quantities $m$ and $M$ are equivalent to $u$ and $\omega$ respectively. Consequently, we first write the equation in equivalent form with the new variables and linearize the system \eqref{1.1} around equilibrium state $(\rho_*, 0, n_*, 0)$, yielding:

%To get a refined estimate that is a higher decay rate which play a crucial role in the proof of optimal decay rate for the nonlinear terms, as we do in \eqref{4.1}, we can write the nonlinear term G into two parts which consists of the conserved form part and the non-conserved one that is an essential difficulty. In order to cancel out the non-conserved part, we define $ m := \rho u, M := n\omega $ so that we can make full use of some symmetric form, i.e., the relaxation drag force term have the opposite sign, of the system. Thanks to the upper and lower bound of $\rho$ and n, m and M are equivalent to u and $\omega$ respectively. We linearized the system \eqref{1.1} around the equilibrium state $(\rho_*, 0, n_*,0)$ and get
\begin{equation}\label{3.1}
\left\{\begin{array}{ll}
(\rho - \rho_*)_t + m_x = 0,&\\
m_t + p'(\rho_*)(\rho-\rho_*)_x - \rho_*M + n_*m = -(mu)_x - (p'(\rho) - p'(\rho_*))(\rho - \rho_*)_x &\\
+ (\rho - \rho_*)M - (n-n_* )m,&\\
(n - n_*)_t + M_x = 0,&\\
M_t + (n-n_*)_x - M_{xx} - n_*m + \rho_* M = -(M\omega)_x - (n_x\omega)_x + (n- n_*)m &\\
- (\rho -\rho_* )M.
\end{array} \right.
\end{equation}
Without loss of generality, we take $\rho_* = 1, n_* = 1 $ and $\gamma=1$ in \eqref{3.1} and consider the linear system
\begin{equation}\label{3.2}
\left\{\begin{array}{ll}
(\rho-\rho_*)_t + m_x = 0,&\\
m_t + (\rho-\rho_*)_x - M + m = 0 &\\
(n-n_*)_t + M_x = 0,&\\
M_t + (n-n_*)_x - M_{xx} - m + M = 0.
\end{array} \right.
\end{equation}
The Fourier transform of  \eqref{3.2} yields $ \partial_t\widehat{U}(\xi,t) = A(\xi)\widehat{U}(\xi,t)$, with$$ \widehat{U}(\xi, t) = (\widehat{\rho-\rho_*},\widehat{ m}, \widehat{n-n_*}, \widehat{M} )^\top $$ and 
$$ A(\xi) = 
\left(\begin{array}{cccc}
0 & -i \xi & 0 & 0 \\
-i \xi & -1 & 0 & 1 \\
0 & 0 & 0 & -i \xi \\
0 & 1 & -i \xi & -1 -\xi^2
\end{array}\right).
$$ 
Now let us analyze the spectrum of $ A(\xi) $. The characteristic equation of $ A(\xi) $ is given by 
\begin{equation}
\begin{aligned} 
\det|\lambda I - A(\xi)| = \lambda^{4} + (\xi^{2} + 2)\lambda^{3} + 3\xi^{2}\lambda^{2} + \xi^{2}(\xi^{2} + 2)\lambda + \xi^{4}
\end{aligned}
\end{equation}
and denote by $ \lambda_i(\xi) (1\leq i \leq 4) $ the eigenvalues of matrix $A(\xi)$, by a direct computation, we obtain
\begin{lma}\label{lma 3.1}
	\begin{enumerate}
	\item There exist positive constants $ r_1 \leq r_2 $	such that $ \lambda_i(\xi) (1\leq i \leq 4) $ has the Taylor series expansion 
	\begin{equation}\label{eq:eigenvalue}
	\begin{aligned} 
		& \lambda_1 = -\frac{1}{2}\xi^{2} + o(\xi^{2}),\\
		& \lambda_2 = i\xi -\frac{1}{4}\xi^{2} + o(\xi^{2}),\\
		& \lambda_3 = -i\xi -\frac{1}{4}\xi^{2} + o(\xi^{2}),\\
		& \lambda_4 = -2 + o(\xi^{2})
	\end{aligned}
	\end{equation}
	for $ |\xi| \leq r_1 $,
	and 
	\begin{equation}
	\begin{aligned} 
		& \lambda_1 = -\xi^{2} + o(\frac{1}{\xi}), \\
		& \lambda_2 = i\xi - \frac{1}{2} + o(\frac{1}{\xi}), \\
		& \lambda_2 = - i\xi - \frac{1}{2} + o(\frac{1}{\xi}), \\
		& \lambda_2 = - 1 + o(\frac{1}{\xi})
	\end{aligned}
	\end{equation}
	for $ |\xi| \geq r_2 $.
	\item The matrix exponential $  e^{tA(\xi)} $has the spectral resolution 
	$$ e^{tA(\xi)} = \sum_{j=1}^{4} e^{t\lambda_j(\xi)}P_j,$$  
	\item $ P_j (1 \leq j \leq 4)$ has the estimate 
	$$ \||P_j\|| \leq C $$
	for $ |\xi| \leq r_1 $ and $ |\xi | \geq r_2 $, where $ \||\cdot\|| $ denotes the matrix norm.  
	\item There exists a positive constant $ \beta_1 $ such that for $ |\xi| \leq r_1 $,
	$$ Re \lambda_j(\xi) \leq - \beta_1|\xi|^{2} \quad (1 \leq j \leq 4 ), $$ 
	\item There exists a positive constant $ \beta_2 $ such that for $ |\xi| \geq r_2 $,
	$$ Re \lambda_j(\xi) \leq - \beta_2 \quad (1 \leq j \leq 4 ). $$ 
	\end{enumerate}
\end{lma}

%According to Lemma \ref{lma 3.1}, the operator $A(\xi)$ possesses simplicity and can be represented through the following spectral decomposition \cite{k1976}:
By Lemma \ref{lma 3.1}, $A(\xi)$ is simple and has the following  spectral representation \cite{k1976}
\begin{equation}\label{3.006}
A(\xi) = \sum_{j=1}^{4}\lambda_j(\xi)P_j(\xi),
\end{equation}
where $P_j$ is called the eigenprojection for the eigenvalue $\lambda_j(\xi)$ of $A(\xi)$ and has the following property
\begin{equation}\label{3.070}
P_j(\xi)P_k(\xi) = \delta_{jk}P_j(\xi),\quad \sum_{j=1}^{4}P_j(\xi) = I.
\end{equation}
Taking Taylor expansions at $\xi = 0$ and by \eqref{3.006} we get
\begin{equation}\label{3.080}
A(\xi) = \sum_{j=1}^{4}\left[\lambda_j(0) + \lambda^{\prime}_j(0)\xi + \cdots\right] \left[P_j(0) + P^{\prime}j(0)\xi + \cdots \right],
\end{equation}
we denote $P_{j0} :=P_j(0)$.

Comparing the constant terms on the both sides of \eqref{3.080} yields
\begin{equation*}
A(0) = \sum_{j=1}^{4}\lambda_j(0)P_{j0}.
\end{equation*}
Taking Taylor expansions in \eqref{3.070}, and then compare the constant terms to obtain
\begin{equation}\label{3.09}
P_{j0}P_{k0} = \delta_{jk}P_{j0},\quad \sum_{j=1}^{4}P_{j0} = I.
\end{equation}
A direct computation for the constant matrix $A(0)$ gives its eigenvalues:
$$\lambda_{10} = \lambda_{20} = \lambda_{30} = 0, \lambda_{40} = -2.$$
Then $\sum_{j=1}^{3}P_{j0}$ is the eigenprojection of $A(0)$ corresponding to the eigenvalue zero, it is obvious that 
\begin{equation}\label{3.0010}
P_0 := \sum_{j=1}^{3}P_{j0} = 
\left(\begin{array}{cccc}
1 & 0 & 0 & 0 \\
0 & \frac{1}{2} & 0 & \frac{1}{2} \\
0 & 0 & 1 & 0 \\
0 & \frac{1}{2} & 0 & \frac{1}{2}
\end{array}\right).
\end{equation}

Note that by \eqref{3.09}
\begin{equation}\label{3.0011}
P_{j0}P_0 = P_0P_{j0} = P_{j0}, \quad 1\leq j \leq 3.
\end{equation}

\begin{lma}
If $ h\in L^{1}(\mathbb{R}) $ and $ \partial_x^{k}h \in L^{2}(\mathbb{R})$, then 
\begin{equation}\label{3.06}
\|e^{At}(i\xi)^{k}\hat{h}(\xi)\| \leq C(1+t)^{-\frac{1}{4} - \frac{k}{2}}\|h\|_{L^{1}} + Ce^{-ct}\|\partial_x^{k}h\|,
\end{equation}
where $C$ and $c$ are positive constants. If in addition, $h$ takes the form 
\begin{equation}\label{3.007}
h = 
\left(\begin{array}{c}
0 \\
g \\
0 \\
-g
\end{array}\right),
\end{equation}
then 
\begin{equation}\label{3.07}
\|e^{At}(i\xi)^{k}\hat{h}(\xi)\| \leq C(1+t)^{-\frac{3}{4} - \frac{k}{2}}\|h\|_{L^{1}} + Ce^{-ct}\|\partial_x^{k}h\|.
\end{equation}
\end{lma}
\begin{proof}
By Lemma \ref{lma 3.1}, we have 
\begin{equation}\label{3.08}
	\begin{aligned} 
		\|e^{At}(i\xi)^{k}\hat{h}(\xi)\|^{2}  = & \int_{\mathbb{R}} |e^{At}(i\xi)^{k}\hat{h}(\xi)|^{2} d\xi   \\
		=& \int_{\mathbb{R}} \|\sum                                                                                                                                                                                                                                                                                                                                                                                                                                                                                                                                                                                                                                                                                                                                                                                                                                                                                                                                                                                                                                                                                                                                                                                                                                                                                                                                                                                                                                                                                                                                                                                                                                                                                                                                                                                                                                                                                              _{i=1}^{4}e^{t\lambda_i(\xi)}P_i(\xi)\|^{2}|\xi|^{2k}|\hat{h(\xi)}|^{2} d\xi  \\
		\leq&\int_{|\xi|\leq r_1} e^{-2\beta_1|\xi|^{2}t}|\xi|^{2k}|\hat{h}(\xi)|^{2} d\xi + \int_{|\xi|\geq r_2} e^{-2\beta_2t}|(i\xi)^{k}\hat{h}(\xi)|^{2} d\xi  \\
		\leq&C\|\hat{h}\|_{\infty}^{2}\int_{|\xi|\leq r_1} e^{-2\beta_1|\xi|^{2}t}|\xi|^{2k}| d\xi + Ce^{-2\beta_2t}\| (i\xi)^{k}\hat{h}(\xi)\|^{2}   \\
		\leq&C(1+t)^{-\frac{1}{2}-k}\|h\|_{L^{1}}^{2} + Ce^{-2\beta_2t}\|\partial_x^{k}h\|^{2}.
	\end{aligned}
\end{equation}
Taking square root on both sides yields \eqref{3.06}.

If $h$ further satisfies \eqref{3.007}, we refine the integral over $|\xi| \leq r_1 $ to obtain \eqref{3.07} as follows. Owing to 
$$ e^{tA(\xi)} = \sum_{j=1}^{4} e^{t\lambda_j(\xi)}P_j(\xi), $$
where $ \lambda_j(\xi) $ and $ P_j(\xi) $ are holomorphic at $\xi = 0 $. Thus for $|\xi| \leq r_1$ with $r_1$ small enough, taking Taylor expansions of $P_j(\xi)$, $1 \leq j \leq 3$, and using \eqref{3.0010}, \eqref{3.0011} and \eqref{3.007}, we have 
\begin{equation*}
\begin{aligned} 
e^{tA(\xi)}(i\xi)^{k}\hat{h}(\xi) & = \sum_{j=1}^{3}e^{\lambda_j(\xi)t}[P_{j0} + O(|\xi|)](i\xi)^{k}\hat{h}(\xi) + e^{\lambda_4(\xi)t}P_4(\xi)(i\xi)^{k}\hat{h}(\xi)  \\
& = \sum_{j=1}^{3}e^{\lambda_j(\xi)t}O(|\xi|)(i\xi)^{k}\hat{h}(\xi) + e^{\lambda_4(\xi)t}P_4(\xi)(i\xi)^{k}\hat{h}(\xi).
\end{aligned}
\end{equation*}
This implies
\begin{equation}
|e^{tA(\xi)}(i\xi)^{k}\hat{h}(\xi)|  \leq C\big( \sum_{j=1}^{3}e^{Re{\lambda_j(\xi)}t}|\xi|^{k+1}|\hat{h}(\xi)| + e^{Re{\lambda_4(\xi)}t}P_4(\xi)|i\xi|^{k}|\hat{h}(\xi)|\big).
\end{equation}
By virtue of \eqref{eq:eigenvalue}, we also have
\begin{equation}\label{3.010}
Re{\lambda_4(\xi)} \leq \frac{1}{2}\lambda_4(0) \leq -c_0 
\end{equation}
for small $\xi$, where $c_0$ is a positive constant. Using Lemma \ref{lma 3.1} and \eqref{3.010}, we refine the integral over $|\xi|\leq r_1$ as 
\begin{equation}\label{3.011}
\begin{aligned} 
\int_{|\xi|\leq r_1} |e^{tA(\xi)}(i\xi)^{k}\hat{h}(\xi)|^{2} d\xi \leq&C\int_{|\xi|\leq r_1} (e^{-2\beta_1\xi^{2}t}|\xi|^{2k+2} + e^{-2c_0t})|\hat{h}(\xi)|^{2} d\xi  \\
\leq&C(1+t)^{-\frac{3}{2}-k}\|\hat{h}\|_{\infty}^{2} \leq C(1+t)^{-\frac{3}{2}-k}\|h\|_{L^{1}}^{2}.
\end{aligned}
\end{equation}
Replacing the corresponding integral in \eqref{3.08} by \eqref{3.011} gives \eqref{3.07}.
\end{proof}
\subsection{Time-weighted energy estimates}
%In this subsection, we embark on deriving the time-weighted energy estimate. To initiate this process, we introduce the subsequent notation for $m=0, 1, 2$ and $ t \geq 0$:
In this subsection, using the time-weighted energy estimate, we derive decay rates for the nonlinear system. The non-optimal decay rates for higher derivatives obtained in this section support us to obtain the optimal ones given in Theorem \ref{thm 1.3}. The new idea in this subsection is that we assume optimal decay rates for lower derivatives, and perform weighted energy estimate to obtain rough estimates on higher derivatives of the solution. In section \ref{section 4}, we use these estimates and spectral analysis to obtain optimal ones for lower derivatives.

We first introduce the following notations for $m=0, 1, 2$ and $ t \geq 0$:
\begin{equation}\label{Nm}
	N_m^{2}(t) = \sup\limits_{\tau\in[0,t]}(1+ \tau)^{\frac{1}{2}+m}\|\partial^{m}_x(\rho-\rho_*, u, n-n_*, \omega)\|^{2}.
\end{equation}

\begin{pro}\label{thm 3.1}
Let $ (\rho_*, 0, n_*, 0) $ be the constant equilibrium state of \eqref{1.1} and  $ (\rho_0-\rho_*,u_0, n_0 - n_*, \omega_0) \in H^{4}(\mathbb{R}) $ satisfies the assmuption in Theorem \ref{thm 1.3}. If further $N_0$, $ N_1$ is small and $ N_2$ is bounded, then the solution of \eqref{1.1}-\eqref{1.2} \rm{{(}}or \eqref{2.2}-\eqref{2.3}\rm{{)}} given in Theorem \ref{thm 1.2} has the following estimates:
\begin{equation}
\begin{aligned} 
& \|(\rho_x, u_x, n_x, \omega_x )(t,\cdot)\|  \leq C \epsilon_0(1 + t)^{-\frac{1}{2}},  \\
& \int_{0}^{t} (1+\tau)\left(\|\omega_{xx}(\tau,\cdot)\|^{2} + \|(u-\omega)_x(\tau,\cdot)\|^{2}\right) d\tau \leq C\epsilon_0^{2}.
\end{aligned}
\end{equation}
\end{pro}
\begin{proof}
For $ t \geq 0 $, we define 
$$
\begin{array}{ll}
\di M_1^{2}(t) := \sup\limits_{0 \leq \tau \leq t} \bigg\{(1+\tau)\|(\rho_x, u_x, n_x, \omega_x)(\tau,\cdot)\|^{2} \\[4mm]
\di \qquad\qquad + \int_{0}^{t}(1 + \tau)\left(\|\omega_{xx}(\tau,\cdot)\|^{2} + \|(u - \omega)_x(\tau,\cdot)\|^{2}\right) d\tau\bigg\}.
\end{array}
$$
Our goal is to prove 
$$M_1^{2}(t) \leq C\epsilon_0^{2},$$
where $C$ is a positive constant. In what follows, we assume that $M_1(t), N_0(t), N_1(t)$ are small and $N_2(t)$ is bounded.

We first use $ N_0, N_1$ and $ N_2 $ to express some $ L^{\infty} $ norms needed in this section. Using Lemma \ref{lma 1.2}, we obtain 
\begin{equation}\label{3.13}
\begin{aligned} 
& \|\rho - \rho_*\|_{\infty} \leq C\|\rho - \rho_*\|^{\frac{1}{2}}\|(\rho - \rho_*)_x\|^{\frac{1}{2}} \leq CN_0^{\frac{1}{2}}N_1^{\frac{1}{2}}(1 + t)^{-\frac{1}{2}}, \\
& \|n - n_*\|_{\infty} \leq C\|n - n_*\|^{\frac{1}{2}}\|(n - n_*)_x\|^{\frac{1}{2}} \leq CN_0^{\frac{1}{2}}N_1^{\frac{1}{2}}(1 + t)^{-\frac{1}{2}},  \\
& \|u_x\|_{\infty} \leq C\|u_x\|^{\frac{1}{2}}\|u_{xx}\|^{\frac{1}{2}} \leq CN_1^{\frac{1}{2}}N_2^{\frac{1}{2}}(1 + t)^{-1}, \\
& \|v_x\|_{\infty} \leq C\|v_x\|^{\frac{1}{2}}\|v_{xx}\|^{\frac{1}{2}} \leq CN_1^{\frac{1}{2}}N_2^{\frac{1}{2}}(1 + t)^{-1}, \\
& \|n_x\|_{\infty} \leq C\|n_x\|^{\frac{1}{2}}\|n_{xx}\|^{\frac{1}{2}} \leq CN_1^{\frac{1}{2}}N_2^{\frac{1}{2}}(1 + t)^{-1}, \\
& \|\omega_x\|_{\infty} \leq C\|\omega_x\|^{\frac{1}{2}}\|\omega_{xx}\|^{\frac{1}{2}} \leq CN_1^{\frac{1}{2}}N_2^{\frac{1}{2}}(1 + t)^{-1}. \\
\end{aligned}
\end{equation}
 
Next, we estimate $ \|u - \omega\|_{\infty} $, note that 
\begin{equation}\label{3.6}
\begin{aligned} 
u_t + \sigma_*v_x - n_*(\omega - u) = - uu_x - \frac{\gamma - 1}{2}vv_x + (n - n_*)(\omega - u),
\end{aligned}
\end{equation}
\begin{equation}\label{3.7}
\begin{aligned} 
\omega_t - \omega_{xx} - \rho_*(u - \omega) = -\omega\omega_x - \frac{1}{n}n_x - \frac{1}{n}n_x\omega_x + (\rho - \rho_*)(u - \omega),
\end{aligned}
\end{equation}
subtracting \eqref{3.7} from \eqref{3.6} yields, 
\begin{equation}\label{3.16}
\begin{aligned} 
& (u - \omega)_t + (\rho_* + n_*)(u - \omega) = -\sigma_*v_x -uu_x - \frac{\gamma - 1}{2}vv_x \\
& -\omega_{xx} + \omega\omega_x + \frac{1}{n}n_x + \frac{1}{n}n_x\omega_x + (n_* -n + \rho_* - \rho_{0*})(u - \omega) :=R,
\end{aligned}
\end{equation}
thus, we obtain
\begin{equation}\label{3.17}
u - \omega = e^{-(\rho_* + n_*)t}(u_0 - \omega_0) + \int_{0}^{t}  e^{-(\rho_* + n_*)(t - \tau)}R(x,\tau) d\tau.
\end{equation}
Therefore, denote $ c := \rho_* + n_* >0 $, we have
$$ \|u - \omega\|_{\infty} \leq  e^{-ct}\|(u_0 - \omega_0)\|_{\infty} + \int_{0}^{t}  e^{-c(t - \tau)}\|R(\tau,x)\|_{\infty} d\tau. $$
Owing to the expression of $R$ and Lemma \ref{lma 1.2}, 
\begin{equation}
\begin{aligned} 
 \|R(t,\cdot)\|_{\infty}  \leq & C\big(\|v_x\|_{\infty} + \|u_x\|_{\infty} + \|n_x\|_{\infty} + \|\omega_{xx}\|_{\infty} + \|n-n_*\|_{\infty}\|u-\omega\|_{\infty}  \\
& + \|\rho - \rho_*\|_{\infty}\|u-\omega\|_{\infty}\big)  \\
\leq&C\big(\|v_x\|^{\frac{1}{2}}\|v_{xx}\|^{\frac{1}{2}} + \|u_x\|^{\frac{1}{2}}\|u_{xx}\|^{\frac{1}{2}} + \|n_x\|^{\frac{1}{2}}\|n_{xx}\|^{\frac{1}{2}} \\
& + \|\omega_{xx}\|^{\frac{1}{2}}\|\omega_{xxx}\|^{\frac{1}{2}} + \|n-n_*\|^{\frac{1}{2}}\|(n-n_*)_x\|^{\frac{1}{2}}\|u-\omega\|_{\infty} \\
&					 + \|\rho - \rho_*\|^{\frac{1}{2}}\|(\rho - \rho_*)_x\|^{\frac{1}{2}}\|u-\omega\|_{\infty}\big),  
\end{aligned}
\end{equation}
hence,
\begin{equation}\label{3.12}
\begin{aligned} 
& \int_{0}^{t} e^{-c(t-\tau)}\|R(\tau,\cdot)\|_{\infty} d\tau \\
\leq&CN_1^{\frac{1}{2}}N_2^{\frac{1}{2}} \int_{0}^{t} e^{-c(t-\tau)}(1 + \tau)^{-1} d\tau + C\int_{0}^{t} e^{-c(t-\tau)}\|\omega_{xx}\|^{\frac{1}{2}}\|\omega_{xxx}\|^{\frac{1}{2}} d\tau  \\
& + CN_0^{\frac{1}{2}}N_1^{\frac{1}{2}} \int_{0}^{t} e^{-c(t-\tau)}(1 + \tau)^{-\frac{1}{2}} \|(u -\omega)\|_{\infty} d\tau.
\end{aligned}
\end{equation}
We estimate terms on the right hand side of \eqref{3.12}. A direct calculation gives
\begin{equation*}
\begin{aligned} 
& CN_1^{\frac{1}{2}}N_2^{\frac{1}{2}}\int_{0}^{t} e^{-c(t-\tau)}(1 + \tau)^{-1} d\tau \leq CN_1^{\frac{1}{2}}N_2^{\frac{1}{2}}(1 + t)^{-1},  
\end{aligned}
\end{equation*}
\begin{equation*}
	\begin{aligned} 
& \int_{0}^{t} e^{-c(t-\tau)}\|\omega_{xx}\|^{\frac{1}{2}}\|\omega_{xxx}\|^{\frac{1}{2}} d\tau = \int_{0}^{t} e^{-c(t-\tau)}\|\omega_{xx}\|^{\frac{3}{8}}\|\omega_{xx}\|^{\frac{1}{8}}\|\omega_{xxx}\|^{\frac{1}{2}} d\tau  \\
\leq&N_2^{\frac{3}{8}} \int_{0}^{t} e^{-c(t-\tau)}(1 + \tau)^{-\frac{31}{32}}\big((1 + \tau)^{2}\|\omega_{xxx}\|^{2}\big)^{\frac{1}{4}} \|\omega_{xx}\|^{\frac{1}{8}} d\tau \\
\leq&CN_2^{\frac{3}{8}} \big(\int_{0}^{t} (1 + \tau)^{2}\|\omega_{xxx}\|^{2} d\tau \big)^{\frac{1}{4}} \big(\int_{0}^{t} e^{-c(t-\tau)}(1 + \tau)^{-\frac{31}{24}}\|\omega_{xx}\|^{\frac{1}{6}}  d\tau\big)^{\frac{3}{4}}  \\
\leq &CN_2^{\frac{3}{8}}M_2^{\frac{1}{2}} \big(\int_{0}^{t} e^{-c(t-\tau)}(1 + \tau)^{-\frac{33}{24}}\big((1 + \tau)\|\omega_{xx}\|^{2} \big)^{\frac{1}{12}} d\tau \big)^{\frac{3}{4}}  \\
\leq&CN_2^{\frac{3}{8}}M_2^{\frac{1}{2}} \big(\int_{0}^{t} (1 + \tau)\|\omega_{xx}\|^{2} d\tau \big)^{\frac{3}{48}} \big( \int_{0}^{t} e^{- c(t - \tau)}(1 + \tau)^{-\frac{3}{2}} d\tau\big)^{\frac{11}{16}} \\
\leq&CN_2^{\frac{3}{8}}M_2^{\frac{1}{2}}M_1^{\frac{3}{24}}(1 + \tau)^{-\frac{33}{32}},  \\
\end{aligned}
\end{equation*}
where we have used H\"older's inequality two times.
\begin{equation*}
\begin{aligned} 
& N_0^{\frac{1}{2}}N_1^{\frac{1}{2}}\int_{0}^{t} e^{-c(t-\tau)}(1 + \tau)^{-\frac{1}{2}} \|u -\omega\|_{\infty} d\tau  \\
\leq&CN_0^{\frac{1}{2}}N_1^{\frac{1}{2}}(1 + t)^{-\frac{3}{2}} \sup\limits_{0 \leq \tau \leq t}\left[(1 + \tau)\|u - \omega\|_{\infty} \right].
\end{aligned}
\end{equation*}
Plugging the above estimates into \eqref{3.12}, we obtain
\begin{equation}
\begin{aligned} 
& \int_{0}^{t} e^{-c(t-\tau)} \|R(\tau,\cdot)\|_{\infty} d\tau \\
\leq&CN_1^{\frac{1}{2}}N_2^{\frac{1}{2}}(1 + t)^{-1} + CN_2^{\frac{3}{8}}M_2^{\frac{1}{2}}M_1^{\frac{3}{24}}(1 + \tau)^{-\frac{33}{32}}  \\
& + CN_0^{\frac{1}{2}}N_1^{\frac{1}{2}}(1 + t)^{-\frac{3}{2}} \sup\limits_{0 \leq \tau \leq t}\left[(1 + \tau)\|u - \omega\|_{\infty}\right].
\end{aligned}
\end{equation}
Therefore, we obtain
\begin{equation}
\begin{aligned} 
& \|u - \omega\|_{\infty} \leq e^{-ct}\|(u_0 - \omega_0)\|_{\infty} + \int_{0}^{t}  e^{-c(t - \tau)}\|R(\tau,\cdot)\|_{\infty} d\tau \\
\leq&e^{-ct}\|(u_0 - \omega_0)\|_{\infty} +  CN_1^{\frac{1}{2}}N_2^{\frac{1}{2}}(1 + t)^{-1} + CN_2^{\frac{3}{8}}M_2^{\frac{1}{2}}M_1^{\frac{3}{24}}(1 + \tau)^{-\frac{33}{32}}  \\
& + CN_0^{\frac{1}{2}}N_1^{\frac{1}{2}}(1 + t)^{-\frac{3}{2}} \sup\limits_{0 \leq \tau \leq t}\left[(1 + \tau)\|u - \omega\|_{\infty}\right],
\end{aligned}
\end{equation}
which implies
\begin{equation}
\begin{aligned} 
\sup\limits_{0 \leq \tau \leq t}\left[(1 + \tau)\|u - \omega\|_{\infty} \right] \leq C\left[\|(u_0 - \omega_0)\|_{\infty} + N_1^{\frac{1}{2}}N_2^{\frac{1}{2}} + N_2^{\frac{3}{8}}M_2^{\frac{1}{2}}M_1^{\frac{3}{24}}\right]
\end{aligned}
\end{equation}
for small $ N_0, N_1 $. 
Using Sobolev embedding inequality, we have 
\begin{equation}\label{3.025}
\begin{aligned} 
\|(u - \omega)(\tau,\cdot)\|_{\infty} \leq C\left(\|(u_0 - \omega_0)\|_{H^{1}} + N_1^{\frac{1}{2}}N_2^{\frac{1}{2}} + N_2^{\frac{3}{8}}M_2^{\frac{1}{2}}M_1^{\frac{3}{24}}\right)(1 + \tau)^{-1}.
\end{aligned}
\end{equation}

Now, we start the weighted energy estimate. Differentiating $\eqref{2.2}_1, \eqref{2.2}_2$ with $x$, multiplying them by $ \rho v_x, \rho u_x $ respectively,
differentiating $ \eqref{2.2}_3 $ and $ \eqref{2.2}_4 $ with $x$, multiplying $\eqref{2.2}_3$ by $ n_x $ and then divided by $n$, multiplying $\eqref{2.2}_4$ by $ n\omega_x $, then adding them together, we can obtain
\begin{equation}
\begin{aligned}
& \frac{1}{2}\rho (v_x^{2} + u_x^{2})_t + \frac{1}{2n}(n_x^{2})_t + \frac{1}{2}n(\omega_x^{2})_t + \sigma_{*}\rho (u_xv_x)_x + (n_x\omega_{xx})_x + \rho n(u-\omega)_x^{2}   \\
 =& -\frac{1}{2}\rho u_x^{3} -\frac{\gamma}{2}\rho u_xv_x^{2}
+ \rho n_x(\omega - u)u_x - \frac{1}{n}n_xn_{xx}\omega - \frac{2}{n}n_x^{2}\omega_x - n\omega_x^{3}  \\
&- n\omega\omega_x\omega_{xx} + \frac{1}{n}\omega_xn_x^{2} + \omega_x^{2}n_{xx} - \frac{1}{n}\omega_x^{2}n_x^{2} + n_x\omega_x\omega_{xx} + n\omega_x\omega{xxx}   \\
& + n\rho_x(u-\omega)\omega_x.  
\end{aligned}
\end{equation}

We replace the time variable by $ \tau$, multiply the equation by the weighted function $ (1 + \tau) $, and integrate the result over $ \mathbb{R}\times [0, t] $. After integrating by parts, we have 
\begin{equation}\label{3.23}
\begin{aligned}
& \frac{1}{2}\int_{\mathbb{R}} (1+t)(\rho v_x^{2} + \rho u_x^{2} + \frac{1}{n}n_x^{2} + n\omega_x^{2}) dx \\
& + \int_{0}^{t}\int_{\mathbb{R}} (1+\tau)(n\omega_{xx}^{2} + \rho n(u-\omega)_x^{2}) dxd\tau   \\
=& \frac{1}{2}\int_{\mathbb{R}} \big(\rho_0 u_{0x}^{2} + \rho_0 v_{0x}^{2} + \frac{1}{n_0}n_{0x}^{2} + n_0\omega_{0x}^{2}\big) dx + \int_{0}^{t}\int_{\mathbb{R}} (1+\tau)\bigg[-\rho u_x^{3}   \\
& - \frac{1}{2}(\gamma+1)\rho u_xv_x^{2} - \frac{1}{2}\rho_xuu_x^{2}   - \frac{1}{2}\rho_xuv_x^{2} + \rho n_x(\omega - u)u_x - n\omega_x^{3}   \\
& +\frac{1}{2}n_x\omega\omega_x^{2} - \frac{1}{n}\omega_x^{2}n_x^{2} 
-2n_x\omega_x\omega_{xx} + n\rho_x(u-\omega)\omega_x\bigg] dxd\tau   \\
& + \frac{1}{2}\int_{0}^{t}\int_{\mathbb{R}}(\rho u_x^{2} + \rho v_x^{2} + \frac{1}{n}n_x^{2} + n\omega_x^{2}) dxd\tau. 
\end{aligned}
\end{equation}
By Lemma \ref{lma 1.2} and \eqref{3.13},
\begin{equation*}
\begin{aligned}
& \int_{0}^{t}\int_{\mathbb{R}} (1+\tau)\rho u_x^{3} dxd\tau  \\
\leq &\int_{0}^{t} (1+\tau) \|\rho\|_{\infty}\|u_x\|_{\infty}\|u_x\|^{2} d\tau 
\leq CN_1^{\frac{1}{2}}N_2^{\frac{1}{2}}\int_{0}^{t} \|u_x\|^{2} d\tau. 
\end{aligned}
\end{equation*}
Similarly,
\begin{equation*}
\begin{aligned}
& \int_{0}^{t}\int_{\mathbb{R}}  (1+\tau)\rho u_xv_x^{2} dxd\tau \leq CN_1^{\frac{1}{2}}N_2^{\frac{1}{2}}\int_{0}^{t} \|v_x\|^{2} d\tau,  \\
& \int_{0}^{t}\int_{\mathbb{R}} (1+\tau)\rho n_x(\omega - u)u_x dxd\tau  \\
 \leq&  C\big(\|(u_0 - \omega_0)\|_{\infty} + N_1^{\frac{1}{2}}N_2^{\frac{1}{2}}  + N_2^{\frac{3}{8}}M_2^{\frac{1}{2}}M_1^{\frac{3}{24}}\big) \int_{0}^{t}\|(u_x,n_x)\|^{2} d\tau,  \\
& \int_{0}^{t}\int_{\mathbb{R}}   (1+t)n\omega_x^{3} dxd\tau \leq CN_1^{\frac{1}{2}}N_2^{\frac{1}{2}}\int_{0}^{t} \|\omega_x\|^{2} d\tau,  \\
& \int_{0}^{t}\int_{\mathbb{R}}  (1+\tau)\rho_xuu_x^{2} dxd\tau \leq CN_0^{\frac{1}{2}}N_1N_2^{\frac{1}{2}}M_1^{2},  \\
& \int_{0}^{t}\int_{\mathbb{R}}  (1+t)\rho_xuv_x^{2} dxd\tau  \leq CN_0^{\frac{1}{2}}N_1N_2^{\frac{1}{2}}M_1^{2},   \\
& \int_{0}^{t}\int_{\mathbb{R}}  (1+t)n_x\omega\omega_x^{2} dxd\tau \leq CN_0^{\frac{1}{2}}N_1N_2^{\frac{1}{2}}M_1^{2},   \\
& \int_{0}^{t}\int_{\mathbb{R}}  (1+\tau)\frac{1}{n}\omega_x^{2}n_x^{2} dxd\tau \leq CN_1N_2M_1^{2}.  \\
\end{aligned}
\end{equation*}
By virtue of Cauchy-Schwarz inequality,  Lemma \ref{lma 1.2} and Young's inequality
\begin{equation*}
\begin{aligned}
& \int_{0}^{t}\int_{\mathbb{R}}  (1+\tau)n_x\omega_x\omega_{xx} dxd\tau   \\
%\leq&C\int_{0}^{t} (1+\tau)\|\omega_x\|_{\infty}\|n_x\|\|\omega_{xx}\| d\tau  \\
\leq&C\int_{0}^{t} (1+\tau)\|\omega_x\|^{\frac{1}{2}}\|n_x\|\|\omega_{xx}\|^{\frac{3}{2}} d\tau  \\
\leq&\delta \int_{0}^{t} (1+\tau)\|\omega_{xx}\|^{2} d\tau + C\int_{0}^{t}\|\omega_x\|^{2}\|n_x\|^{4} d\tau \\
\leq&\delta \int_{0}^{t} (1+\tau)\|\omega_{xx}\|^{2} d\tau + CM_1^{4}M_1^{2}.
\end{aligned}
\end{equation*}
Using Cauchy-Schwarz inequality and \eqref{3.025},
\begin{equation*}
\begin{aligned}
& \int_{0}^{t}\int_{\mathbb{R}} (1+\tau)n\rho_x(u-\omega)\omega_x dxd\tau \\
\leq &\int_{0}^{t} (1+\tau)\|n\|_{\infty}\|(u-\omega)\|_{\infty}\|\rho_x\|\|\omega_x\| d\tau \\
\leq&C\left(\|(u_0 - \omega_0)\|_{H^{1}} + N_1^{\frac{1}{2}}N_2^{\frac{1}{2}} + N_2^{\frac{3}{8}}M_2^{\frac{1}{2}}M_1^{\frac{3}{24}}\right) \int_{0}^{t} \|(v_x,\omega_x)\|^{2} d\tau.
\end{aligned}
\end{equation*}
Plugging these estimates into \eqref{3.23} and choosing $ \delta $ sufficiently small, then combining with Theorem \ref{thm 1.2} we can obtain
\begin{equation}\label{3.37}
\begin{aligned}
M_1^{2 } \leq C\epsilon_0^{2} + C\left(N_1N_2 + N_0^{\frac{1}{2}}N_1N_2^{\frac{1}{2}} + M_1^{4}\right)M_1^{2}, 
\end{aligned}
\end{equation}
which implies
\begin{equation}
\begin{aligned}
\left[1-C(N_1N_2 + N_0^{\frac{1}{2}}N_1N_2^{\frac{1}{2}} + M_1^{4})\right]M_1^{2} \leq C\epsilon_0^{2}.
\end{aligned}
\end{equation}
Therefore, if $ N_0, N_1, M_1 $ are small and $ N_2 $ is bounded, we get 
\begin{equation}
\begin{aligned}
M_1^{2}(t) \leq C\epsilon_0^{2}.
\end{aligned}
\end{equation}
\end{proof}

%Subsequently, we present five propositions. Prosositon \ref{thm 3.2} is the time-weighted energy estimate for the second derivative and its proof is similar to Proposition \ref{thm 3.1}.
%Propositions \ref{3 order} and \ref{4 order} serve as foundation in proving Propositions \ref{thm 3.3} and \ref{thm 3.4}. In conjunction with Propositions \ref{thm 3.1} and \ref{thm 3.2}, these propositions are important for the demonstration of Proposition \ref{pro 4.1}. The methodologies of their proofs are similar to those of Lemmas \ref{thm 2.5}, \ref{thm 2.6}, Propositions \ref{thm 3.1}, and \ref{thm 3.2} respectively. We omit the details.

Subsequently, we present five propositions. Proposition \ref{thm 3.2} is the time-weighted energy estimate for the second derivative and its proof is similar to Proposition \ref{thm 3.1}. Propositions \ref{3 order} and \ref{4 order} are the base to proof of Propositions \ref{thm 3.3} and \ref{thm 3.4},  which together with Proposition \ref{thm 3.1}, \ref{thm 3.2} will be needed in the proof of Proposition \ref{pro 4.1}. Their proofs are similar to Lemmas \ref{thm 2.5}, \ref{thm 2.6}, Propositions \ref{thm 3.1} and \ref{thm 3.2} respectively. we omit the details. 
\begin{pro}\label{thm 3.2}
Let $ (\rho_*, 0, n_*, 0) $ be the constant equilibrium state of \eqref{1.1} and  $ (\rho_0-\rho_*,u_0, n_0 - n_*, \omega_0) \in H^{4}(\mathbb{R}) $ satisfies the assumptions in Theorem \ref{thm 1.3}. If further $N_0$, $ N_1$ is small and $ N_2$ is bounded, the solution of \eqref{1.1}-\eqref{1.2} {\rm{(}}or \eqref{2.2}-\eqref{2.3} {\rm{)}} given in Theorem \ref{thm 1.2} has the following estimates:
\begin{equation}
\begin{aligned}
&  \|(\rho_{xx}, u_{xx}, n_{xx}, \omega_{xx})(t,\cdot)\|   \leq C(\epsilon_0 + {\mathcal{E}_0})(1+t)^{-1}  \\
& \int_{0}^{t} (1+\tau)^{2} \left(\|\omega_{xxx}(\tau,\cdot)\|^{2} + \|(u-\omega)_{xx}(\tau,\cdot)\|^{2}\right) d\tau \leq C (\epsilon_0^{2} +{\mathcal{E}^2_0}).
\end{aligned}
\end{equation}
\end{pro}

\begin{pro}\label{3 order}
Let $ (\rho_*, 0, n_*, 0) $ be the constant equilibrium state of \eqref{1.1} and  $ (\rho_0-\rho_*,u_0, n_0 - n_*, \omega_0) \in H^{4}(\mathbb{R}) $ satisfies the assumptions in Theorem \ref{thm 1.3}. If further $N_0$, $ N_1$ is small and $ N_2$ is bounded, the solution of \eqref{1.1}-\eqref{1.2}  {\rm{(}}or \eqref{2.2}-\eqref{2.3} {\rm{)}} given in Theorem \ref{thm 1.2} has the following estimates:
\begin{equation}
\begin{aligned}
&  \|\partial^{3}_x(\rho, u, n, \omega)(t,\cdot) \|^{2} + \|\partial^{2}_x(\rho_t, u_t)(t,\cdot)\|^{2}   \\
+ & \int_{0}^{t}\left(\|\big(\partial^{3}_x(\rho, u, n), \partial^{2}_xu_{\tau}\big)(\tau,\cdot)\|^{2} + \|\partial^{4}_x\omega(\tau,\cdot)\|^{2}\right) d\tau    
\leq C (\epsilon_0^{2} +{\mathcal{E}^2_0}) .
\end{aligned}
\end{equation}
\end{pro}

\begin{pro}\label{4 order}
Let $ (\rho_*, 0, n_*, 0) $ be the constant equilibrium state of \eqref{1.1} and  $ (\rho_0-\rho_*,u_0, n_0 - n_*, \omega_0) \in H^{4}(\mathbb{R}) $ satisfies the assumptions in Theorem \ref{thm 1.3}. If further $N_0$, $ N_1$ is small and $ N_2$ is bounded, the solution of \eqref{1.1}-\eqref{1.2}  {\rm{(}}or \eqref{2.2}-\eqref{2.3} {\rm{)}} given in Theorem \ref{thm 1.2} has the following estimates:
\begin{equation}
\begin{aligned}
& \|\partial^{4}_x(\rho, u, n, \omega)(t,\cdot)\|^{2} + \|\partial^{3}_x(\rho_t, u_t)(t,\cdot)\|^{2} \\+ & \int_{0}^{t}\left(\|\left(\partial^{4}_x(\rho, u, n), \partial^{3}_xu_{\tau}\right)(\tau,\cdot)\|^{2} + \|\partial^{5}_x\omega(\tau,\cdot)\|^{2}\right) d\tau    
\leq C (\epsilon_0^{2} +{\mathcal{E}^2_0}) .
\end{aligned}
\end{equation}
\end{pro}

\begin{pro}\label{thm 3.3}
Let $ (\rho_*, 0, n_*, 0) $ be the constant equilibrium state of \eqref{1.1} and  $ (\rho_0-\rho_*,u_0, n_0 - n_*, \omega_0) \in H^{4}(\mathbb{R}) $ satisfies the assumptions in Theorem \ref{thm 1.3}. If further $N_0$, $ N_1$ is small and $ N_2$ is bounded, the solution of \eqref{1.1}-\eqref{1.2}  {\rm{(}}or \eqref{2.2}-\eqref{2.3} {\rm{)}} given in Theorem \ref{thm 1.2} has the following estimates:
\begin{equation}
\begin{aligned}
&  \|\partial^{3}_x(\rho, u, n, \omega)(t,\cdot)\|   \leq C(\epsilon_0 + {\mathcal{E}_0})(1+t)^{-\frac{3}{2}}  \\
& \int_{0}^{t} (1+\tau)^{3}\left(\|\partial^{4}_x\omega(\tau,\cdot)\|^{2} + \|\partial^{3}_x(u-\omega)(\tau,\cdot)\|^{2}\right) d\tau    \\
\leq&C (\epsilon_0^{2} +{\mathcal{E}^2_0}) .
\end{aligned}
\end{equation}
\end{pro}
\begin{pro}\label{thm 3.4}
Let $ (\rho_*, 0, n_*, 0) $ be the constant equilibrium state of \eqref{1.1} and  $ (\rho_0-\rho_*,u_0, n_0 - n_*, \omega_0) \in H^{4}(\mathbb{R}) $ satisfies the assumptions in Theorem \ref{thm 1.3}. If further $N_0$, $ N_1$ is small and $ N_2$ is bounded, the solution of \eqref{1.1}-\eqref{1.2}  {\rm{(}}or \eqref{2.2}-\eqref{2.3} {\rm{)}} given in Theorem \ref{thm 1.2} has the following estimates:
\begin{equation}
\begin{aligned}
&  \|\partial^{4}_x(\rho, u, n, \omega)(t,\cdot)\|   \leq C(\epsilon_0 + {\mathcal{E}_0})(1+t)^{-2}  \\
& \int_{0}^{t} (1+\tau)^{4}\left(\|\partial^{5}_x\omega(\tau,\cdot)\|^{2} + \|\partial^{4}_x(u-\omega)(\tau,\cdot)\|^{2}\right) d\tau    \\
\leq&C (\epsilon_0^{2} +{\mathcal{E}^2_0}).
\end{aligned}
\end{equation}
\end{pro}\

\section{Optimal time decay rates}\label{section 4}

%Now we give a complete proof to the optimal time decay rate in Theorem \ref{thm 1.3}, based on the spectral analysis for the linearized system and the time-weighted energy estimates to the original nonlinear system.

%Within this section, we leverage the outcomes acquired in Section \ref{section 3} to deduce the optimal decay rate of the solution $(\rho, u, n, \omega)$ for the Cauchy problem \eqref{1.1}-\eqref{1.2}. This solution converges towards the constant equilibrium state $(\rho_*, 0, n_*, 0)$.

In this section, we finish the nonlinear analysis to prove our main result, Theorem \ref{thm 1.3}, to obtain optimal time decay rates of the solution $(\rho, u, n, \omega)$ of  \eqref{1.1}-\eqref{1.2} to the constant equilibrium state $(\rho_*, 0, n_*, 0)$. We use the Duhamel's principle to express the solution of nonlinear system and use the results obtained in section \ref{section 3} to make a priori estimate.

Recall $N^{2}_m$ defined in \eqref{Nm},
\begin{equation}
	N_m^{2}(t) = \sup\limits_{\tau\in[0,t]}(1+ \tau)^{\frac{1}{2}+m}\|\partial^{m}_x(\rho-\rho_*, u, n-n_*, \omega)\|^{2},
\end{equation}
where $m=0, 1, 2$ and $ t \geq 0$. By a standard continuity argument, to prove \eqref{1.6} in Theorem \ref{thm 1.3}, we only need to prove the following proposition.

%we used the results obtained in section \ref{section 3} to get the optimal decay rates of the solution $(\rho, u, n, \omega)$ to  the Cauchy problem \eqref{1.1}, \eqref{1.2} which tend toward the constant equilibrium state $(\rho_*, 0, n_*, 0).$
\begin{pro}\label{pro 4.1}
Under the hypotheses of Theorem \ref{thm 1.3}, if $N_0, N_1$ is bounded by a small positive constant, $N_2$ is bounded by a constant, which are independent of $T \geq 0$, then
$$ N_1 \leq CF_0$$ 
and 
$$ N_2 \leq C(F_0 + {\mathcal{E}_0}). $$
\end{pro}
\begin{proof}
We can write \eqref{3.1} in the form
\begin{equation}\label{4.1}
U_t = AU + G(U),
\end{equation}
where 
\begin{equation}
\begin{aligned}
G  = &\left(\begin{array}{c}
0\\
-(mu)_x - \big(P'(\rho)-P'(\rho_*)\big)(\rho - \rho_*)_x \\
0 \\
-(M\omega)_x - (n_x\omega)_x
\end{array}\right)    \\
&+ 
 \left(\begin{array}{c}
0\\
(\rho -\rho_* )M - (n- n_* )m \\
0 \\
-(\rho -\rho_* )M + (n-n_*)m
\end{array}\right),
\end{aligned}
\end{equation}
or in the Fourier transform of \eqref{4.1}
\begin{equation}\label{4.2}
\hat{U}_t = A(\xi)\hat{U} + \hat{G}(U).
\end{equation}
The solution of \eqref{4.2} is 
\begin{equation}\label{4.4}
\begin{aligned} 
\hat{U} = e^{tA(\xi)}\hat{U}(0) + \int_{0}^{t}e^{(t-\tau)A}\hat{G}(U)(\xi,\tau) d\tau.
\end{aligned}
\end{equation}
By Planchel theorem, \eqref{4.4}, and triangle inequality, we have
\begin{equation}\label{4.5}
\begin{aligned}  
\|\partial_x^{k}U\|  = &\|(i\xi)^{k}\hat{U}\| \\
\leq&\|(i\xi)^{k}e^{tA(\xi)}\hat{U}(0)\| + \int_{0}^{t} \|(i\xi)^{k}e^{(t-\tau)A}\hat{G}(U)(\xi,\tau)\| d\tau,
\end{aligned}
\end{equation}
from \eqref{3.06} we have 
\begin{equation}\label{4.6}
\begin{aligned}  
\|(i\xi)^{k}e^{tA(\xi)}\hat{U}(0)\| \leq C(1+t)^{-\frac{1}{4} - \frac{k}{2}}\left(\|U(0)\|_{L^{1}} + \|\partial_x^{k}U(0)\|\right).
\end{aligned}
\end{equation}
Similarly, using \eqref{4.2}, \eqref{3.06} and \eqref{3.07}, we obtain
\begin{equation}\label{4.7}
\int_{0}^{t} \|(i\xi)^{k}e^{(t-\tau)A}\hat{G}(U)(\xi,\tau)\| d\tau \leq I_1 + I_2 + I_3,
\end{equation}
where 
\begin{equation}\label{4.8}
\begin{aligned}  
I_1 =&C\int_{0}^{\frac{t}{2}}(1+t-\tau)^{-\frac{1}{4}-\frac{k+1}{2}}\bigg(\|mu\|_{L^{1}} + \|(\rho - \rho_*)^{2}\|_{L^{1}} + \|M\omega\|_{L^{1}} + \|n_x\omega\|_{L^{1}}  \\
& + \|(\rho -\rho_*)M\|_{L^{1}} + \|(n -n_*)m\|_{L^{1}}\bigg) d\tau,  \\
 I_2 = &C\int_{\frac{t}{2}}^{t}(1+t-\tau)^{-\frac{3}{4}}\bigg[\|\partial_x^{k}(mu)\|_{L^{1}} + \|\partial_x^{k}(\rho - \rho_*)^{2}\|_{L^{1}} + \|\partial_x^{k}(M\omega)\|_{L^{1}}  \\
& + \|\partial_x^{k}(n_x\omega)\|_{L^{1}}  + \|\partial_x^{k}\big((\rho -\rho_*)M\big)\|_{L^{1}} + \|\partial_x^{k}\big((n -n_*)m\big)\|_{L^{1}}\bigg]d\tau, \\
 I_3 = &C\int_{0}^{t} e^{-c(t-\tau)}\bigg(\|\partial_x^{k+1}(mu)\| + \|\partial_x^{k+1}(\rho - \rho_*)^{2}\| + \|\partial_x^{k+1}(M\omega)\|  \\
& + \|\partial_x^{k+1}(n_x\omega)\| + \|\partial_x^{k}((\rho -\rho_*)M)\| + \|\partial_x^{k}((n -n_*)m)\|\bigg) d\tau.
\end{aligned}
\end{equation}
	\noindent$\bullet$ (Estimate on $I_1$): 
First, we have
\begin{equation*}
\begin{aligned}  
& \|mu\|_{L^{1}} \leq \|m\|\|u\| \leq CN_0^{2}(1+\tau)^{-\frac{1}{2}},  \\
& \|(\rho - \rho_*)^{2}\|_{L^{1}} \leq CN_0^{2}(1+\tau)^{-\frac{1}{2}},  \\
& \|M\omega\|_{L^{1}} \leq CN_0^{2}(1+\tau)^{-\frac{1}{2}},  \\
& \|(\rho -\rho_*)M\|_{L^{1}} \leq CN_0^{2}(1+\tau)^{-\frac{1}{2}},  \\
& \|(n -n_*)m\|_{L^{1}} \leq CN_0^{2}(1+\tau)^{-\frac{1}{2}},  \\
& \|n_x\omega\|_{L^{1}} \leq \|n_x\|\|\omega\| \leq CN_0N_1(1+\tau)^{-1}.
\end{aligned}
\end{equation*}
Substituting the above estimates into \eqref{4.8} yields 
\begin{equation}\label{4.9}
\begin{aligned} 
I_1 \leq &(N_0^{2} + N_0N_1)\int_{0}^{\frac{t}{2}} (1+t-\tau)^{-\frac{3}{4}- \frac{k}{2}}(1+\tau)^{-\frac{1}{2}} d\tau  \\
\leq& C(N_0^{2} + N_0N_1)(1+t)^{-\frac{1}{4}-\frac{k}{2}}.
\end{aligned}
\end{equation}
For the estimates of $I_2$ and $I_3$, we take $k=0,1,2$ separately and have the estimates as follows:

\noindent $\bullet$ (Estimates of $I_2, I_3$ for $k=0$):

Applying the above estimates gives
\begin{equation}\label{4.10}
I_2 \leq  C(N_0^{2} + N_0N_1)(1+t)^{-\frac{1}{4}}.
\end{equation}
For $I_3$, 
\begin{equation*}
\begin{aligned}  
\|(mu)_x\| \leq&\|m_xu\| + \|mu_x\|  \\
\leq&C\|u\|_{\infty}\|m_x\| + \|m\|_{\infty}\|u_x\|  \\
\leq&CN_0^{\frac{1}{2}}N_1^{\frac{3}{2}}(1+\tau)^{-\frac{5}{4}}.
\end{aligned}
\end{equation*}
Similarly,
\begin{equation*}
\begin{aligned}  
\|(\rho - \rho_*)^{2}_x\| 
\leq&CN_0^{\frac{1}{2}}N_1^{\frac{3}{2}}(1+\tau)^{-\frac{5}{4}}, \\
\|(M\omega)_x\| 
\leq& CN_0^{\frac{1}{2}}N_1^{\frac{3}{2}}(1+\tau)^{-\frac{5}{4}}, \\
\|(\rho -\rho_*)M\| \leq&\|\rho - \rho_*\|_{\infty}\|M\|  
\leq CN_0^{\frac{3}{2}}N_1^{\frac{1}{2}}(1+\tau)^{-\frac{3}{4}},  \\
\|(n -n_*)m\| \leq&\|n-n_*\|_{\infty}\|m\|  
\leq CN_0^{\frac{3}{2}}N_1^{\frac{1}{2}}(1+\tau)^{-\frac{3}{4}}.
\end{aligned}
\end{equation*}
By Lemma \ref{lma 1.2}, \eqref{3.13}, Proposition \ref{thm 3.3}, Lemma \ref{thm 2.2} and \ref{thm 2.5}, we have 
\begin{equation*}
\begin{aligned}  
\|(n_x\omega)_x\| \leq&\|n_{xx}\omega\| + \|n_x\omega_x\|  \\
\leq&\|\omega\|_{\infty}\|n_{xx}\| + \|n_x\|_{\infty}\|\omega_x\|  \\
\leq&C\|\omega\|^{\frac{1}{2}}\|\omega_x\|^{\frac{1}{2}}\|n_x\|^{\frac{1}{2}}\|n_{xxx}\|^{\frac{1}{2}} + C\|n_x\|^{\frac{1}{2}}\|n_{xx}\|^{\frac{1}{2}}\|\omega_x\|  \\
\leq&CN_0^{\frac{1}{2}}N_1^{\frac{1}{2}}\epsilon_0^{\frac{1}{2}}(\epsilon_0 + {\mathcal{E}_0})^{\frac{1}{2}}(1+\tau)^{-\frac{3}{2}} + CN_1\epsilon_0^{\frac{1}{2}}(\epsilon_0 + {\mathcal{E}_0})^{\frac{1}{2}}(1+\tau)^{-\frac{3}{2}}. 
\end{aligned}
\end{equation*}
Substituting the above estimates into \eqref{4.8} yields 
\begin{equation}\label{4.11}
\begin{aligned}  
I_3 \leq &C(N_0^{\frac{1}{2}}N_1^{\frac{3}{2}} + N_0^{\frac{1}{2}}N_1^{\frac{1}{2}}\epsilon_0^{\frac{1}{2}}(\epsilon_0 + {\mathcal{E}_0})^{\frac{1}{2}} + N_1\epsilon_0^{\frac{1}{2}}(\epsilon_0 + {\mathcal{E}_0})^{\frac{1}{2}}  \\
& + N_0^{\frac{3}{2}}N_1^{\frac{1}{2}})\int_{0}^{t} e^{-c(t-\tau)}(1+\tau)^{-\frac{3}{4}} d\tau   \\
\leq& C\left(N_0^{\frac{1}{2}}N_1^{\frac{3}{2}} + N_0^{\frac{1}{2}}N_1^{\frac{1}{2}}\epsilon_0^{\frac{1}{2}}(\epsilon_0 + {\mathcal{E}_0})^{\frac{1}{2}} + N_1\epsilon_0^{\frac{1}{2}}(\epsilon_0 + {\mathcal{E}_0})^{\frac{1}{2}} + N_0^{\frac{3}{2}}N_1^{\frac{1}{2}}\right)(1+t)^{-\frac{3}{4}}.   
\end{aligned}
\end{equation}
Combining \eqref{4.7}, \eqref{4.9}, \eqref{4.10} and \eqref{4.11}, we have 
\begin{equation}\label{4.12}
\begin{aligned}  
& \int_{0}^{t} \|e^{(t-\tau)A}\hat{G}(U)(\xi,\tau)\| d\tau \\
\leq&C\left(N_0^{2} + N_1^{2} + \epsilon_0^{\frac{1}{2}}(\epsilon_0 + {\mathcal{E}_0})^{\frac{1}{2}}N_0 + \epsilon_0^{\frac{1}{2}}(\epsilon_0 + {\mathcal{E}_0})^{\frac{1}{2}}N_1\right)(1+t)^{-\frac{1}{4}}. 
\end{aligned}
\end{equation}
Substituting \eqref{4.6} and \eqref{4.12} into \eqref{4.5} implies that 
\begin{equation*}
\begin{aligned}  
\|U\| \leq C\left(F_0 + N_0^{2} + N_1^{2} + \epsilon_0^{\frac{1}{2}}(\epsilon_0 + {\mathcal{E}_0})^{\frac{1}{2}}N_0 + \epsilon_0^{\frac{1}{2}}(\epsilon_0 + {\mathcal{E}_0})^{\frac{1}{2}}N_1\right)(1+t)^{-\frac{1}{4}}. 
\end{aligned}
\end{equation*}
Equivalently, 
\begin{equation*}
\begin{aligned}  
(1+t)^{\frac{1}{4}} \|U\| \leq C\left(F_0 + N_0^{2} + N_1^{2} + \epsilon_0^{\frac{1}{2}}(\epsilon_0 + {\mathcal{E}_0})^{\frac{1}{2}}N_0 + \epsilon_0^{\frac{1}{2}}(\epsilon_0 + {\mathcal{E}_0})^{\frac{1}{2}}N_1\right).
\end{aligned}
\end{equation*}
Taking supremum, we have 
\begin{equation}\label{4.13}
\begin{aligned}  
N_0 \leq C\left(F_0 + N_0^{2} + N_1^{2}\right) + C\epsilon_0^{\frac{1}{2}}(\epsilon_0 + {\mathcal{E}_0})^{\frac{1}{2}}N_0 + C\epsilon_0^{\frac{1}{2}}(\epsilon_0 + {\mathcal{E}_0})^{\frac{1}{2}}N_1.
\end{aligned}
\end{equation}
\noindent $\bullet$ (Estimates of $I_2, I_3$ for $k=1$):
\begin{equation}\label{4.14}
\begin{aligned}  
 I_2 = &C\int_{\frac{t}{2}}^{t}(1+t-\tau)^{-\frac{3}{4}}\big(\|(mu)_x\|_{L^{1}} + \|[(\rho - \rho_*)^{2}]_x\|_{L^{1}} + \|(M\omega)_x\|_{L^{1}}  \\
& + \|(n_x\omega)_x\|_{L^{1}}  + \|((\rho -\rho_*)M)_x\|_{L^{1}} + \|((n -n_*)m)_x\|_{L^{1}}\big) d\tau.  
\end{aligned}
\end{equation}
Using Cauchy-Schwarz inequality and \eqref{3.13} gives
\begin{equation*}
\begin{aligned} 
& \|(mu)_x\|_{L^{1}} \leq  \|m_xu\|_{L^{1}} + \|mu_x\|_{L^{1}}  \\
\leq& \|m_x\|\|u\| + \|m\|\|u_x\|  \leq CN_0N_1(1+\tau)^{-1}.
\end{aligned}
\end{equation*}
Similarly, we have 
\begin{equation*}
\begin{aligned}  
& \|[(\rho - \rho_*)^{2}]_x\|_{L^{1}} \leq CN_0N_1(1+\tau)^{-1},  \\
& \|(M\omega)_x\|_{L^{1}} \leq CN_0N_1(1+\tau)^{-1},   \\
& \|((\rho -\rho_*)M)_x\|_{L^{1}} \leq CN_0N_1(1+\tau)^{-1},   \\
& \|((n -n_*)m)_x\|_{L^{1}} \leq CN_0N_1(1+\tau)^{-1}.
\end{aligned}
\end{equation*}
By Cauchy-Schwarz inequality, Lemmas \ref{lma 1.2}, \ref{thm 2.2} and Proposition \ref{thm 3.3}, we have
\begin{equation*}
\begin{aligned}  
\|(n_x\omega)_x\|_{L^{1}} \leq&\|n_{xx}\omega\|_{L^{1}} + \|n_x\omega_x\|_{L^{1}}  \\
\leq&\|n_{xx}\|\|\omega\| + \|n_x\|\|\omega_x\|   \\
\leq&\|n_x\|^{\frac{1}{2}} \|n_{xxx}\|^{\frac{1}{2}}\|\omega\| + \|n_x\|\|\omega_x\|   \\
\leq&N_0\epsilon_0^{\frac{1}{2}}(\epsilon_0 + {\mathcal{E}_0})^{\frac{1}{2}}(1+\tau)^{-\frac{5}{4}} + CN_1^{2}(1+\tau)^{-\frac{3}{2}}.
\end{aligned}
\end{equation*}
Substituting the above estimates into \eqref{4.14} yields
\begin{equation}\label{4.15}
\begin{aligned}  
 I_2 \leq &C\left[N_0N_1 + N_1^{2} +  N_0\epsilon_0^{\frac{1}{2}}(\epsilon_0 + {\mathcal{E}_0})^{\frac{1}{2}} \right]\int_{\frac{t}{2}}^{t}(1+t-\tau)^{-\frac{3}{4}}(1+\tau)^{-1} d\tau  \\
\leq&C\left[N_0^{2} + N_1^{2} + \epsilon_0^{\frac{1}{2}}(\epsilon_0 + {\mathcal{E}_0})^{\frac{1}{2}} N_0\right](1+t)^{-\frac{3}{4}}.
\end{aligned}
\end{equation}
When $k=1$,
\begin{equation}\label{4.16}
\begin{aligned}  
 I_3 = &C\int_{0}^{t} e^{-c(t-\tau)}\bigg[\|(mu)_{xx}\| + \|\big((\rho - \rho_*)^{2}\big)_{xx}\|  \\
& + \|(M\omega)_{xx}\| + \|(n_x\omega)_{xx}\| + \|((\rho -\rho_*)M)_x\| + \|((n -n_*)m)_x\|\bigg] d\tau.
\end{aligned}
\end{equation}
We estimate terms on the right hand side of \eqref{4.16}.
It follows from Lemma \ref{lma 1.2}, Proposition \ref{thm 3.3} and \eqref{3.13} that
\begin{equation}
\begin{aligned}  
\|(mu)_{xx}\|  \leq &\|m_{xx}u\| + 2\|m_xu_x\| + \|mu_{xx}\|   \\
\leq&\|u\|_{\infty}\|m_{xx}\| + 2\|u_x\|_{\infty}\|m_x\| + \|m\|_{\infty}\|u_{xx}\|  \\
\leq&C\|u\|^{\frac{1}{2}}\|u_x\|^{\frac{1}{2}}\|m_x\|^{\frac{1}{2}}\|m_{xxx}\|^{\frac{1}{2}} + C\|u_x\|^{\frac{1}{2}}\|u_{xx}\|^{\frac{1}{2}}\|m_x\|  \\
& + C\|m\|^{\frac{1}{2}}\|m_x\|^{\frac{1}{2}}\|u_x\|^{\frac{1}{2}}\|u_{xxx}\|^{\frac{1}{2}}   \\
\leq&C\epsilon_0^{\frac{1}{2}}(\epsilon_0 + {\mathcal{E}_0})^{\frac{1}{2}}N_0^{\frac{1}{2}}N_1^{\frac{1}{2}}(1+\tau)^{-\frac{3}{2}} + C\epsilon_0^{\frac{1}{2}}(\epsilon_0 + {\mathcal{E}_0})^{\frac{1}{2}}N_1(1+\tau)^{-\frac{3}{2}}.
\end{aligned}
\end{equation}
Similarly, 
\begin{equation*}
\begin{aligned}  
\|\big((\rho - \rho_*)^{2}\big)_{xx}\| &\leq C\epsilon_0^{\frac{1}{2}}(\epsilon_0 + {\mathcal{E}_0})^{\frac{1}{2}}N_0^{\frac{1}{2}}N_1^{\frac{1}{2}}(1+\tau)^{-\frac{3}{2}} + C\epsilon_0^{\frac{1}{2}}(\epsilon_0 + {\mathcal{E}_0})^{\frac{1}{2}}N_1(1+\tau)^{-\frac{3}{2}}   \\
\|(M\omega)_{xx}\| \leq& C\epsilon_0^{\frac{1}{2}}(\epsilon_0 + {\mathcal{E}_0})^{\frac{1}{2}}N_0^{\frac{1}{2}}N_1^{\frac{1}{2}}(1+\tau)^{-\frac{3}{2}} + C\epsilon_0^{\frac{1}{2}}(\epsilon_0 + {\mathcal{E}_0})^{\frac{1}{2}}N_1(1+\tau)^{-\frac{3}{2}}.   
\end{aligned}
\end{equation*}
Combining Lemmas \ref{lma 1.2}, \ref{thm 2.2}, Propositions \ref{thm 3.3} and \ref{thm 3.4}, we deduce that
\begin{equation*}
\begin{aligned} 
\|(n_x\omega)_{xx}\|\leq&\|n_{xxx}\omega\| + 2\|n_{xx}\omega_x\| + \|n_x\omega_{xx}\|   \\
\leq&\|\omega\|_{\infty}\|n_{xxx}\| + \|\omega_x\|_{\infty}\|n_{xx}\| + \|n_x\|_{\infty}\|\omega_{xx}\|  \\
\leq&C\|\omega\|^{\frac{1}{2}}\|\omega_x\|^{\frac{1}{2}}\|n_x\|^{\frac{1}{3}}\|n_{xxxx}\|^{\frac{2}{3}} + C\|\omega_x\|^{\frac{1}{2}}\|\omega_{xx}\|^{\frac{1}{2}}\|n_{xx}\|  \\
& + C\|n_x\|^{\frac{1}{2}}\|n_{xx}\|^{\frac{1}{2}}\omega_{xx}\|  \\
\leq&C\|\omega\|^{\frac{1}{2}}\|\omega_x\|^{\frac{1}{2}}\|n_x\|^{\frac{1}{3}}\|n_{xxxx}\|^{\frac{2}{3}} + C\|\omega_x\|^{\frac{1}{2}}\|\omega_{x}\|^{\frac{1}{4}}\|\omega_{xxx}\|^{\frac{1}{4}}\|n_{x}\|^{\frac{1}{2}}\|n_{xxx}\|^{\frac{1}{2}}  \\
& + C\|n_x\|^{\frac{1}{2}}\|n_{x}\|^{\frac{1}{4}}\|n_{xxx}\|^{\frac{1}{4}}\omega_{x}\|^{\frac{1}{2}}\|\omega_{xxx}\|^{\frac{1}{2}}  \\
\leq&CN_0^{\frac{1}{2}}N_1^{\frac{1}{2}}\epsilon_0^{\frac{1}{3}}(\epsilon_0 + {\mathcal{E}_0})^{\frac{2}{3}}(1+\tau)^{-2} + CN_1\epsilon_0^{\frac{1}{4}}(\epsilon_0 + {\mathcal{E}_0})^{\frac{3}{4}}(1+\tau)^{-2}.
\end{aligned}
\end{equation*}
Applying Lemma \ref{lma 1.2}, \eqref{3.13} and Proposition \ref{thm 3.1} gives
\begin{equation*}
\begin{aligned} 
\|((\rho -\rho_*)M)_x\| \leq&\|(\rho -\rho_*)_xM\| + \|(\rho -\rho_*)M_x\| \\
\leq&\|M\|_{\infty}\|(\rho -\rho_*)_x\| + \|(\rho -\rho_*)\|_{\infty}\|M_x\|  \\
\leq&C\|M\|^{\frac{1}{2}}\|M_x\|^{\frac{1}{2}}\|(\rho -\rho_*)_x\| + C\|(\rho -\rho_*)\|^{\frac{1}{2}}\|(\rho -\rho_*)_x\|^{\frac{1}{2}}\|M_x\|  \\
\leq&CN_0^{\frac{1}{2}}N_1^{\frac{1}{2}}\epsilon_0(1+\tau)^{-1}.
\end{aligned}
\end{equation*}
Similarly, 
\begin{equation*}
\begin{aligned} 
\|((n -n_*)m)_x\| \leq CN_0^{\frac{1}{2}}N_1^{\frac{1}{2}}\epsilon_0(1+\tau)^{-1}.
\end{aligned}
\end{equation*}
Substituting these estimates into \eqref{4.16} yields
\begin{equation}\label{4.18}
\begin{aligned} 
I_3 \leq &C\bigg[\epsilon_0^{\frac{1}{2}}(\epsilon_0 + {\mathcal{E}_0})^{\frac{1}{2}}N_0^{\frac{1}{2}}N_1^{\frac{1}{2}} + \epsilon_0^{\frac{1}{2}}(\epsilon_0 + {\mathcal{E}_0})^{\frac{1}{2}}N_1 +  CN_0^{\frac{1}{2}}N_1^{\frac{1}{2}}\epsilon_0^{\frac{1}{3}}(\epsilon_0 + {\mathcal{E}_0})^{\frac{2}{3}}   \\
& + N_1\epsilon_0^{\frac{1}{4}}(\epsilon_0 + {\mathcal{E}_0})^{\frac{3}{4}} + N_0^{\frac{1}{2}}N_1^{\frac{1}{2}}\epsilon_0\bigg]\int_{0}^{t} e^{-c(t-\tau)}(1+\tau)^{-\frac{3}{4}} d\tau   \\
\leq&C\left[\epsilon_0^{\frac{1}{2}}(\epsilon_0 + {\mathcal{E}_0})^{\frac{1}{2}} + \epsilon_0^{\frac{1}{3}}(\epsilon_0 + {\mathcal{E}_0})^{\frac{2}{3}} + \epsilon_0^{\frac{1}{4}}(\epsilon_0 + {\mathcal{E}_0})^{\frac{3}{4}} + \epsilon_0 \right](N_0 + N_1)(1+t)^{-\frac{3}{4}}.
\end{aligned}
\end{equation}
Putting \eqref{4.7}, \eqref{4.9}, \eqref{4.15}, and \eqref{4.18} together, we have 
\begin{equation}\label{4.19}
\begin{aligned} 
& \int_{0}^{t} \|(i\xi)e^{(t-\tau)A}\hat{G}(U)(\xi,\tau)\| d\tau   \\
\leq&C\bigg[N_0^{2} + N_1^{2} + \left[\epsilon_0^{\frac{1}{2}}(\epsilon_0 + {\mathcal{E}_0})^{\frac{1}{2}} + \epsilon_0^{\frac{1}{3}}(\epsilon_0 + {\mathcal{E}_0})^{\frac{2}{3}} + \epsilon_0^{\frac{1}{4}}(\epsilon_0 + {\mathcal{E}_0})^{\frac{3}{4}} + \epsilon_0\right]  \\
&(N_0 + N_1)\bigg](1+t)^{-\frac{3}{4}}.
\end{aligned}
\end{equation}
Substituting \eqref{4.6}, \eqref{4.19} into \eqref{4.5} gives
\begin{equation*}
\begin{aligned} 
\|U_x\| \leq&C\big[F_0 + N_0^{2} + N_1^{2} + \big(\epsilon_0^{\frac{1}{2}}(\epsilon_0 + {\mathcal{E}_0})^{\frac{1}{2}} + \epsilon_0^{\frac{1}{3}}(\epsilon_0 + {\mathcal{E}_0})^{\frac{2}{3}}  \\
& + \epsilon_0^{\frac{1}{4}}(\epsilon_0 + {\mathcal{E}_0})^{\frac{3}{4}} + \epsilon_0 \big) (N_0 + N_1)\big](1+t)^{-\frac{3}{4}},
\end{aligned}
\end{equation*}
equivalently,
\begin{equation*}
\begin{aligned} 
(1+t)^{\frac{3}{4}}\|U_x\| \leq&C\big[F_0 + N_0^{2} + N_1^{2} + \big(\epsilon_0^{\frac{1}{2}}(\epsilon_0 + {\mathcal{E}_0})^{\frac{1}{2}} + \epsilon_0^{\frac{1}{3}}(\epsilon_0 + {\mathcal{E}_0})^{\frac{2}{3}}  \\
& + \epsilon_0^{\frac{1}{4}}(\epsilon_0 + {\mathcal{E}_0})^{\frac{3}{4}} + \epsilon_0 \big) (N_0 + N_1)\big]
\end{aligned}
\end{equation*}
taking supremum, we have
\begin{equation}\label{4.20}
\begin{aligned} 
N_1 \leq&C\big[F_0 + N_0^{2} + N_1^{2} + \big(\epsilon_0^{\frac{1}{2}}(\epsilon_0 + {\mathcal{E}_0})^{\frac{1}{2}} + \epsilon_0^{\frac{1}{3}}(\epsilon_0 + {\mathcal{E}_0})^{\frac{2}{3}}  \\
& + \epsilon_0^{\frac{1}{4}}(\epsilon_0 + {\mathcal{E}_0})^{\frac{3}{4}} + \epsilon_0 \big) (N_0 + N_1)\big].
\end{aligned}
\end{equation}
Adding \eqref{4.13} with \eqref{4.20}, we obtain
\begin{equation}
\begin{aligned} 
N_0 + N_1 \leq &C\bigg[F_0 + N_0^{2} + N_1^{2} + \big(\epsilon_0^{\frac{1}{2}}(\epsilon_0 + {\mathcal{E}_0})^{\frac{1}{2}} + \epsilon_0^{\frac{1}{3}}(\epsilon_0 + {\mathcal{E}_0})^{\frac{2}{3}}  \\
& + \epsilon_0^{\frac{1}{4}}(\epsilon_0 + {\mathcal{E}_0})^{\frac{3}{4}} + \epsilon_0 \big) (N_0 + N_1)\bigg]  \\
\leq&CF_0 + C(N_0 + N_1)^{2} + \big(\epsilon_0^{\frac{1}{2}}(\epsilon_0 + {\mathcal{E}_0})^{\frac{1}{2}} + \epsilon_0^{\frac{1}{3}}(\epsilon_0 + {\mathcal{E}_0})^{\frac{2}{3}}  \\
& + \epsilon_0^{\frac{1}{4}}(\epsilon_0 + {\mathcal{E}_0})^{\frac{3}{4}} + \epsilon_0 \big) (N_0 + N_1).
\end{aligned}
\end{equation}
Choosing $ \epsilon_0, N_0, N_1 $ suitably small, we arrive at 
\begin{equation}\label{4.22}
\begin{aligned} 
N_0 + N_1 \leq CF_0.
\end{aligned}
\end{equation}
\noindent $\bullet$ (Estimates of $I_2, I_3$ for $k=2$):

Take $ k =2 $, we have 
\begin{equation}\label{4.23}
\begin{aligned} 
 I_2 = &C\int_{\frac{t}{2}}^{t}(1+t-\tau)^{-\frac{3}{4}}\big(\|(mu)_{xx}\|_{L^{1}} + \|((\rho - \rho_*)^{2})_{xx}\|_{L^{1}} + \|(M\omega)_{xx}\|_{L^{1}}  \\
& + \|(n_x\omega)_{xx}\|_{L^{1}}  + \|((\rho -\rho_*)M)_{xx}\|_{L^{1}} + \|((n -n_*)m)_{xx}\|_{L^{1}}\big) d\tau.  \\
\end{aligned}
\end{equation}
By Cauchy-Schwarz inequality and Proposition \ref{thm 3.1} one gets
\begin{equation*}
\begin{aligned} 
\|(mu)_{xx}\|_{L^{1}} \leq&\|m_{xx}u\|_{L^{1}} + 2\|m_xu_x\|_{L^{1}} + \|mu_{xx}\|_{L^{1}}  \\
\leq&\|u\|\|m_{xx}\| + 2\|m_x\|\|u_x\| + \|m\|\|u_{xx}\|   \\
\leq&\epsilon_0N_2(1+\tau)^{-\frac{3}{2}} + C\epsilon_0^{2}(1+
\tau)^{-\frac{3}{2}}.
\end{aligned}
\end{equation*}
Similarly,
\begin{equation*}
\begin{aligned} 
& \|((\rho - \rho_*)^{2})_{xx}\|_{L^{1}} \leq \epsilon_0N_2(1+\tau)^{-\frac{3}{2}} + C\epsilon_0^{2}(1+
\tau)^{-\frac{3}{2}},  \\
& \|(M\omega)_{xx}\|_{L^{1}} \leq \epsilon_0N_2(1+\tau)^{-\frac{3}{2}} + C\epsilon_0^{2}(1+
\tau)^{-\frac{3}{2}},  \\
& \|((\rho -\rho_*)M)_{xx}\|_{L^{1}} \leq \epsilon_0N_2(1+\tau)^{-\frac{3}{2}} + C\epsilon_0^{2}(1+
\tau)^{-\frac{3}{2}},  \\
& \|((n -n_*)m)_{xx}\|_{L^{1}} \leq \epsilon_0N_2(1+\tau)^{-\frac{3}{2}} + C\epsilon_0^{2}(1+
\tau)^{-\frac{3}{2}}. 
\end{aligned}
\end{equation*}
With the aid of Cauchy-Schwarz inequality, Lemma \ref{lma 1.2}, Proposition \ref{thm 3.4} and Theorem \ref{thm 1.2}, we get
\begin{equation*}
\begin{aligned} 
\|(n_x\omega)_{xx}\|_{L^{1}} \leq&\|n_{xxx}\omega\|_{L^{1}} + 2\|n_{xx}\omega_x\|_{L^{1}} + \|n_x\omega_{xx}\|_{L^{1}}  \\
\leq& \|n_{xxx}\|\|\omega\| + 2\|n_{xx}\|\|\omega_x\| + \|n_x\|\|\omega_{xx}\|  \\  
\leq&\|\omega\| \|n_x\|^{\frac{1}{3}}\|n_{xxxx}\|^{\frac{2}{3}}  + 2\|n_{xx}\|\|\omega_x\| + \|n_x\|\|\omega_{xx}\|  \\  
\leq&C\epsilon_0^{\frac{1}{3}}(\epsilon_0 + {\mathcal{E}_0})^{\frac{2}{3}}(1+\tau)^{-\frac{21}{12}} + C\epsilon_0N_2(1+\tau)^{-2}.
\end{aligned}
\end{equation*}
Substituting these estimates into \eqref{4.23} yields
\begin{equation}\label{4.24}
\begin{aligned} 
I_2 \leq&\big( \epsilon_0N_2 + \epsilon_0^{2} + \epsilon_0^{\frac{1}{3}}(\epsilon_0 + {\mathcal{E}_0})^{\frac{2}{3}}\big)\int_{\frac{t}{2}}^{t}(1+t-\tau)^{-\frac{3}{4}}(1+\tau)^{-\frac{5}{4}} d\tau  \\
\leq&\big( \epsilon_0N_2 + \epsilon_0^{2} + \epsilon_0^{\frac{1}{3}}(\epsilon_0 + {\mathcal{E}_0})^{\frac{2}{3}}\big)(1+\tau)^{-\frac{5}{4}}.
\end{aligned}
\end{equation}
For $I_3$, when $k=2$,
\begin{equation}\label{4.25}
\begin{aligned} 
I_3 = &C\int_{0}^{t} e^{-c(t-\tau)}\big(\|(mu)_{xxx}\| + \|\big((\rho - \rho_*)^{2}\big)_{xxx}\|  \\
& + \|(M\omega)_{xxx}\| + \|(n_x\omega)_{xxx}\| + \|((\rho -\rho_*)M)_{xx}\| + \|((n -n_*)m)_{xx}\|\big) d\tau.
\end{aligned}
\end{equation}
Applying Lemma \ref{lma 1.2} and Proposition \ref{thm 3.3} gives
\begin{equation*}
\begin{aligned} 
\|(mu)_{xxx}\| \leq&\|m_{xxx}u\| + 3\|m_{xx}u_x\| + 3\|m_xu_{xx}\| + \|mu_{xxx}\|   \\
\leq&\|u\|_{\infty}\|m_{xxx}\| + 3 \|u_x\|_{\infty}\|m_{xx}\| + 3 \|m_x\|_{\infty}\|u_{xx}\| + \|m\|_{\infty}\|u_{xxx}\|   \\
\leq&C\|u\|^{\frac{1}{2}}\|u_x\|^{\frac{1}{2}}\|m_{xxx}\| + C\|u_x\|^{\frac{1}{2}}\|u_{xx}\|^{\frac{1}{2}}\|m_{xx}\|  \\
&  + C\|m_x\|^{\frac{1}{2}}\|m_{xx}\|^{\frac{1}{2}}\|u_{xx}\| + 
C\|m\|^{\frac{1}{2}}\|m_x\|^{\frac{1}{2}}\|u_{xxx}\|   \\
\leq&C\epsilon_0^{\frac{1}{3}}(\epsilon_0 + {\mathcal{E}_0})^{\frac{2}{3}}N_0^{\frac{1}{2}}N_1^{\frac{1}{2}}(1+\tau)^{-2} + C\epsilon_0^{\frac{1}{2}}(\epsilon_0 + {\mathcal{E}_0})^{\frac{1}{2}}N_1^{\frac{1}{2}}N_2^{\frac{1}{2}}(1+\tau)^{-2}.
\end{aligned}
\end{equation*}
Similarly,
\begin{equation*}
\begin{aligned} 
\|\big((\rho - \rho_*)^{2}\big)_{xxx}\| \leq &C\epsilon_0^{\frac{1}{3}}(\epsilon_0 + {\mathcal{E}_0})^{\frac{2}{3}}N_0^{\frac{1}{2}}N_1^{\frac{1}{2}}(1+\tau)^{-2} + C\epsilon_0^{\frac{1}{2}}(\epsilon_0 + {\mathcal{E}_0})^{\frac{1}{2}}N_1^{\frac{1}{2}}N_2^{\frac{1}{2}}(1+\tau)^{-2}   \\
\|(M\omega)_{xxx}\| \leq&C\epsilon_0^{\frac{1}{3}}(\epsilon_0 + {\mathcal{E}_0})^{\frac{2}{3}}N_0^{\frac{1}{2}}N_1^{\frac{1}{2}}(1+\tau)^{-2} + C\epsilon_0^{\frac{1}{2}}(\epsilon_0 + {\mathcal{E}_0})^{\frac{1}{2}}N_1^{\frac{1}{2}}N_2^{\frac{1}{2}}(1+\tau)^{-2}.   \\
\end{aligned}
\end{equation*}
By Lemma \ref{lma 1.2}, Propositions \ref{thm 3.3} and \ref{thm 3.4}, we can obtain
\begin{equation*}
\begin{aligned} 
\|(n_x\omega)_{xxx}\| \leq&\|n_{xxxx}\omega\| + 3 \|n_{xxx}\omega_x\| + 3\|n_{xx}\omega_{xx}\| + \|n_x\omega_{xxx}\|   \\
\leq&C\|\omega\|^{\frac{1}{2}}\|\omega_x\|^{\frac{1}{2}}\|n_{xxxx}\| + C\|\omega_x\|^{\frac{1}{2}}\|\omega_{xx}\|^{\frac{1}{2}}\|n_{xxx}\|   \\
& + C\|n_{xx}\|^{\frac{1}{2}}\|n_{xxx}\|^{\frac{1}{2}}\|\omega_{xx}\| + C\|n_x\|^{\frac{1}{2}}\|n_{xx}\|^{\frac{1}{2}}\|\omega_{xxx}\|  \\
\leq&\epsilon_0^{\frac{1}{2}}(\epsilon_0 +{\mathcal{E}_0})N_0^{\frac{1}{2}}(1+\tau)^{-\frac{19}{8}} + C\epsilon_0^{\frac{1}{3}}(\epsilon_0 + {\mathcal{E}_0})^{\frac{2}{3}}N_1^{\frac{1}{2}}N_2^{\frac{1}{2}}(1+\tau)^{-\frac{5}{2}}   \\
& + C\epsilon_0^{\frac{1}{4}}(\epsilon_0 + {\mathcal{E}_0})^{\frac{3}{4}}N_2(1+\tau)^{-\frac{5}{2}}.
\end{aligned}
\end{equation*}
By Lemma \ref{lma 1.2}, Theorem \ref{thm 1.2}, Propositions \ref{thm 3.1} and \ref{thm 3.3}, we deduce that
\begin{equation*}
\begin{aligned} 
\|((\rho -\rho_*)M)_{xx}\|
\leq&\|(\rho - \rho_*)_{xx}M\| + 2\|(\rho - \rho _*)_xM_x\| + \|(\rho -\rho_*)M_{xx}\|   \\
\leq&\|M\|_{\infty}\|(\rho - \rho_*)_{xx}\| + 2\|(\rho - \rho_*)_x\|_{\infty}\|M_x\| + \|(\rho -\rho_*)\|_{\infty}\|M_{xx}\|  \\
\leq&C\|M\|^{\frac{1}{2}}\|M_x\|^{\frac{1}{2}}\|(\rho - \rho_*)_{x}\|^{\frac{1}{2}}\|(\rho - \rho_*)_{xxx}\|^{\frac{1}{2}}    \\
& + C\|(\rho - \rho_*)_x\|^{\frac{1}{2}}\|(\rho - \rho_*)_{xx}\|^{\frac{1}{2}}\|M_x\| + C\|(\rho -\rho_*)\|^{\frac{1}{2}}\|(\rho -\rho_*)_x\|^{\frac{1}{2}}\|M_{xx}\|  \\
\leq&C\epsilon_0^{\frac{1}{2}}(\epsilon_0 + {\mathcal{E}_0})^{\frac{1}{2}}N_0^{\frac{1}{2}}N_1^{\frac{1}{2}}(1+\tau)^{-\frac{3}{2}} + C\epsilon_0N_1^{\frac{1}{2}}N_2^{\frac{1}{2}}(1+\tau)^{-\frac{3}{2}}.
\end{aligned}
\end{equation*}
Similarly,
\begin{equation*}
\begin{aligned} 
\|((n -n_*)m)_{xx}  
\leq&C\epsilon_0^{\frac{1}{2}}(\epsilon_0 + {\mathcal{E}_0})^{\frac{1}{2}}N_0^{\frac{1}{2}}N_1^{\frac{1}{2}}(1+\tau)^{-\frac{3}{2}} + C\epsilon_0N_1^{\frac{1}{2}}N_2^{\frac{1}{2}}(1+\tau)^{-\frac{3}{2}}.
\end{aligned}
\end{equation*}
Substituting the above estimates into \eqref{4.25}, we have 
\begin{equation}\label{4.26}
\begin{aligned} 
I_3 \leq&\bigg[\epsilon_0^{\frac{1}{3}}(\epsilon_0 + {\mathcal{E}_0})^{\frac{2}{3}}N_0^{\frac{1}{2}}N_1^{\frac{1}{2}} + \epsilon_0^{\frac{1}{2}}(\epsilon_0 + {\mathcal{E}_0})^{\frac{1}{2}}N_1^{\frac{1}{2}}N_2^{\frac{1}{2}}   \\
& + \epsilon_0^{\frac{1}{2}}(\epsilon_0 +{\mathcal{E}_0})N_0^{\frac{1}{2}} + C\epsilon_0^{\frac{1}{3}}(\epsilon_0 + {\mathcal{E}_0})^{\frac{2}{3}}N_1^{\frac{1}{2}}N_2^{\frac{1}{2}}  \\
& + \epsilon_0^{\frac{1}{2}}(\epsilon_0 + {\mathcal{E}_0})^{\frac{1}{2}}N_0^{\frac{1}{2}}N_1^{\frac{1}{2}} + \epsilon_0N_1^{\frac{1}{2}}N_2^{\frac{1}{2}}\bigg]\int_{0}^{t} e^{-c(t-\tau)}(1+\tau)^{-\frac{5}{4}} d\tau    \\
\leq&\bigg[\left(\epsilon_0 + \epsilon_0^{\frac{1}{3}}(\epsilon_0 + {\mathcal{E}_0})^{\frac{2}{3}} + \epsilon_0^{\frac{1}{2}}(\epsilon_0 + {\mathcal{E}_0})^{\frac{1}{2}} \right)(N_0 + N_1 + N_2)   \\
& + \epsilon_0^{\frac{1}{2}}(\epsilon_0 +{\mathcal{E}_0})N_0^{\frac{1}{2}}\bigg](1+\tau)^{-\frac{5}{4}}.
\end{aligned}
\end{equation}
Combining \eqref{4.7}, \eqref{4.9}, \eqref{4.24} and \eqref{4.26}, we have 
\begin{equation}\label{4.27}
\begin{aligned} 
& \int_{0}^{t} \|(i\xi)^{k}e^{(t-\tau)A}\hat{G(U)}(\xi,\tau)\| d\tau  \\
\leq&\bigg[\epsilon_0^{2} + \epsilon_0^{\frac{1}{3}}(\epsilon_0 + {\mathcal{E}_0})^{\frac{2}{3}} + \left(\epsilon_0 + \epsilon_0^{\frac{1}{3}}(\epsilon_0 + {\mathcal{E}_0})^{\frac{2}{3}} + \epsilon_0^{\frac{1}{2}}(\epsilon_0 + {\mathcal{E}_0})^{\frac{1}{2}} \right)(N_0 + N_1 + N_2)   \\
& + \epsilon_0^{\frac{1}{2}}(\epsilon_0 +{\mathcal{E}_0})N_0^{\frac{1}{2}}\bigg](1+t)^{-\frac{5}{4}}. 
\end{aligned}
\end{equation}
Plugging \eqref{4.6}, \eqref{4.27} into \eqref{4.5}, we obtain 
\begin{equation*}
\begin{aligned} 
\|U_{xx}\| \leq&C\bigg[F_0 + \mathcal{E}_0 + \epsilon_0^{\frac{1}{3}}(\epsilon_0 + {\mathcal{E}_0})^{\frac{2}{3}} + \big(\epsilon_0 + \epsilon_0^{\frac{1}{3}}(\epsilon_0 + {\mathcal{E}_0})^{\frac{2}{3}}  \\
&  + \epsilon_0^{\frac{1}{2}}(\epsilon_0 + {\mathcal{E}_0})^{\frac{1}{2}} \big)(N_0 + N_1 + N_2) + \epsilon_0^{\frac{1}{2}}(\epsilon_0 +{\mathcal{E}_0})N_0^{\frac{1}{2}}\bigg](1+t)^{-\frac{5}{4}}, 
\end{aligned}
\end{equation*}
that is,
\begin{equation*}
\begin{aligned} 
(1+t)^{\frac{5}{4}}  \|U_{xx}\| \leq&C\big[F_0 + \mathcal{E}_0 + \epsilon_0^{\frac{1}{3}}(\epsilon_0 + {\mathcal{E}_0})^{\frac{2}{3}} + \big(\epsilon_0 + \epsilon_0^{\frac{1}{3}}(\epsilon_0 + {\mathcal{E}_0})^{\frac{2}{3}}  \\
&  + \epsilon_0^{\frac{1}{2}}(\epsilon_0 + {\mathcal{E}_0})^{\frac{1}{2}} \big)(N_0 + N_1 + N_2) + \epsilon_0^{\frac{1}{2}}(\epsilon_0 +{\mathcal{E}_0})N_0^{\frac{1}{2}}\big],
\end{aligned}
\end{equation*}
taking supremum yields 
\begin{equation}\label{4.28}
\begin{aligned} 
N_2 \leq&C\bigg[F_0 + \mathcal{E}_0  + \epsilon_0^{\frac{1}{3}}(\epsilon_0 + {\mathcal{E}_0})^{\frac{2}{3}} + \big(\epsilon_0 + \epsilon_0^{\frac{1}{3}}(\epsilon_0 + {\mathcal{E}_0})^{\frac{2}{3}}  \\
&  + \epsilon_0^{\frac{1}{2}}(\epsilon_0 + {\mathcal{E}_0})^{\frac{1}{2}} \big)(N_0 + N_1 + N_2) + \epsilon_0^{\frac{1}{2}}(\epsilon_0 +{\mathcal{E}_0})N_0^{\frac{1}{2}}\bigg].
\end{aligned}
\end{equation}
Combining \eqref{4.22} and \eqref{4.28} yields
\begin{equation}\label{4.29}
\begin{aligned} 
N_0 + N_1 + N_2 \leq &C\big[F_0 + \mathcal{E}_0  + \epsilon_0^{\frac{1}{3}}(\epsilon_0 + {\mathcal{E}_0})^{\frac{2}{3}} + \big(\epsilon_0 + \epsilon_0^{\frac{1}{3}}(\epsilon_0 + {\mathcal{E}_0})^{\frac{2}{3}}  \\
&  + \epsilon_0^{\frac{1}{2}}(\epsilon_0 + {\mathcal{E}_0})^{\frac{1}{2}} \big)(N_0 + N_1 + N_2) + \epsilon_0^{\frac{1}{2}}(\epsilon_0 +{\mathcal{E}_0})N_0^{\frac{1}{2}}\big]  \\
\leq&C\left(F_0 + \mathcal{E}_0 \right)  + C\epsilon_0^{\frac{1}{2}}(\epsilon_0 + {\mathcal{E}_0})^{\frac{1}{2}} (N_0 + N_1 + N_2) \\
& + \epsilon_0^{\frac{1}{2}}(\epsilon_0 +{\mathcal{E}_0})(N_0 + N_1 + N_2)^{\frac{1}{2}}.
\end{aligned}
\end{equation}
Owing to the smallness of $\epsilon_0$ and Young's inequality, we obtain 
\begin{equation}
N_0 + N_1 + N_2 \leq C(F_0 + {\mathcal{E}_0}).
\end{equation}
We now carry out the proof of $\|\partial_x^{k}(u-\omega)\|$ in Theorem \ref{thm 1.3}. For this we recall \eqref{3.17}, which implies that
\begin{equation}\label{4.30}
\|u - \omega\| \leq  e^{-ct}\|(u_0 - \omega_0)\| + \int_{0}^{t}  e^{-c(t - \tau)}\|R(x,\tau)\| d\tau,
\end{equation}
where $c$ is a positive constant and $R$ is defined as in \eqref{3.16}. To estimate $\|R(x,\tau)\|$, by virtue of \eqref{3.13} and \eqref{1.6}, we obtain
\begin{equation*}
\begin{aligned} 
& \|v_x\| \leq CF_0(1+\tau)^{-\frac{3}{4}},  \\
& \|uu_x\| \leq C(F_0 + \mathcal{E}_0)^{2}(1+\tau)^{-\frac{5}{4}},   \\
& \|vv_x\| \leq C(F_0 + \mathcal{E}_0)^{2}(1+\tau)^{-\frac{5}{4}},  \\
& \|\omega\omega_x\| \leq C(F_0 + \mathcal{E}_0)^{2}(1+\tau)^{-\frac{5}{4}},   \\
& \|\omega_{xx}\| \leq C(F_0 + {\mathcal{E}_0})(1+\tau)^{-\frac{5}{4}},   \\
& \|\frac{1}{n}n_x\| \leq CF_0(1+\tau)^{-\frac{3}{4}},  \\
& \|\frac{1}{n}n_x\omega_x\| \leq  C(F_0 + {\mathcal{E}_0})^{2}(1+\tau)^{-\frac{7}{4}},  \\
& \|(n-n_*)(u-\omega)\| \leq CF_0(1+\tau)^{-\frac{1}{2}} \|u-\omega\|,    \\
& \|(\rho -\rho_*)(u-\omega)\| \leq CF_0(1+\tau)^{-\frac{1}{2}} \|u-\omega\|.
\end{aligned}
\end{equation*}
Therefore, we can get 
\begin{equation}\label{4.31}
\begin{aligned} 
 \|R\| \leq &C\left(F_0 + {\mathcal{E}_0} \right)(1+\tau)^{-\frac{3}{4}} + CF_0(1+\tau)^{-\frac{1}{2}} \|u-\omega\|.
\end{aligned}
\end{equation}
Substituting \eqref{4.31} into \eqref{4.30}, we obtain
\begin{equation*}
\begin{aligned} 
\|u - \omega\| \leq & e^{-ct}\|(u_0 - \omega_0)\| + \left(F_0 + {\mathcal{E}_0} \right)(1+t)^{-\frac{3}{4}}  + CF_0(1+t)^{-\frac{5}{4}}\sup\limits_{\tau\in[0,t]}\left[(1+\tau)^{\frac{3}{4}}\|u-\omega\| \right] \\
\leq&C(F_0 + {\mathcal{E}_0} )(1+t)^{-\frac{3}{4}} + CF_0(1+t)^{-\frac{5}{4}}\sup\limits_{\tau\in[0,t]} \left[(1+\tau)^{\frac{3}{4}}\|u-\omega\| \right], \\
\end{aligned}
\end{equation*}
furthermore,
\begin{equation*}
\begin{aligned} 
& (1+t)^{\frac{3}{4}}\|u - \omega\| \leq C(F_0 + {\mathcal{E}_0} ) + CF_0(1+t)^{-\frac{1}{2}}\sup\limits_{\tau\in[0,t]}\left[(1+\tau)^{\frac{3}{4}}\|u-\omega\| \right].
\end{aligned}
\end{equation*}
Because of the smallness of $F_0$, we have
\begin{equation}
\begin{aligned} 
\|u - \omega\|  \leq C(F_0 + {\mathcal{E}_0})(1+t)^{-\frac{3}{4}}.
\end{aligned}
\end{equation}
Similarly,
\begin{equation}\label{4.33}
\begin{aligned} 
\|(u - \omega)_x\| \leq  e^{-ct}\|(u_0 - \omega_0)_x\| + \int_{0}^{t}  e^{-c(t - \tau)}\|R_x(\cdot,\tau)\| d\tau.
\end{aligned}
\end{equation}
To estimate $\|R_x(\tau,\cdot)\|$, by virtue of \eqref{3.13} and \eqref{1.6}, we obtain
\begin{equation*}
\begin{aligned} 
\|v_{xx}\| \leq&C(F_0 + {\mathcal{E}_0})(1+\tau)^{-\frac{5}{4}},  \\
\|(uu_x)_x\| \leq&C(F_0 + {\mathcal{E}_0})^{2}(1+\tau)^{-\frac{7}{4}},  \\
\|(vv_x)_x\| \leq&C(F_0 + {\mathcal{E}_0})^{2}(1+\tau)^{-\frac{7}{4}},  \\
\|(\omega \omega_x)_x\| \leq&C(F_0 + {\mathcal{E}_0})^{2}(1+\tau)^{-\frac{7}{4}},   \\
\|\omega_{xxx}\| \leq&C(F_0 + {\mathcal{E}_0})(1+\tau)^{-\frac{3}{2}},   \\
\|(\frac{1}{n}n_x\omega_x)_x\| \leq&C(F_0 + {\mathcal{E}_0})^{2}(1+\tau)^{-\frac{7}{4}},   \\
\|\big((n-n_*)(u-\omega)\big)_x\| \leq&\|(n-n_*)_x(u-\omega)\| + \|(n-n_*)(u-\omega)_x\|  \\
\leq&C(F_0 + {\mathcal{E}_0})^2(1+\tau)^{-\frac{7}{4}} + CF_0(1+\tau)^{-\frac{1}{2}}\|(u-\omega)_x\|.
\end{aligned} 
\end{equation*}
Therefore, we can get 
\begin{equation}\label{4.34}
\begin{aligned} 
\|R_x(\tau,\cdot)\| \leq&C\left(F_0 + {\mathcal{E}_0} \right)(1+\tau)^{-\frac{5}{4}}+ CF_0(1+\tau)^{-\frac{1}{2}}\|(u-\omega)_x\|.
\end{aligned}
\end{equation}
Plugging \eqref{4.34} into \eqref{4.33} we can obtain
\begin{equation*}
\begin{aligned} 
\|(u - \omega)_x\| \leq&e^{-ct}\|(u_0 - \omega_0)_x\| +  C\left
(F_0 + \mathcal{E}_0 \right)(1+t)^{-\frac{5}{4}}   \\
& + CF_0(1+\tau)^{-\frac{7}{4}}\sup\limits_{0\leq \tau \leq t}\left[(1+t)^{\frac{5}{4}}\|(u-\omega)_x\|\right],
\end{aligned}
\end{equation*}
which implies that 
\begin{equation*}
\begin{aligned} 
& (1+t)^{\frac{5}{4}}\|(u - \omega)_x\| \\
\leq&C\left(F_0 + \mathcal{E}_0\right)  + CF_0(1+t)^{-\frac{1}{2}}\sup\limits_{0\leq \tau \leq t}\left[(1+\tau)^{\frac{5}{4}}\|(u-\omega)_x\|\right].
\end{aligned}
\end{equation*}
Due to the smallness of $F_0$, we have
\begin{equation}
\begin{aligned} 
\|(u - \omega)_x\| \leq & C\left(F_0 + {\mathcal{E}_0}\right)(1+t)^{-\frac{5}{4}},
\end{aligned}
\end{equation}
this complete the proof of Theorem \ref{thm 1.3}.
\end{proof}

{\bf Conflict of Interest:} The authors declared that they have no conflicts of interest to this work.

\bibliographystyle{abbrv}  %类似plain，将月份全拼改为缩写，更显紧凑.
\bibliography{ref}

\end{document}